\documentclass[12pt]{article}
\usepackage{epsfig}
\usepackage{amsfonts}
\usepackage{amssymb}
\usepackage{amsmath}
\usepackage{amsthm}
\usepackage{latexsym}
\usepackage{graphicx}
\usepackage{bm}
\usepackage{tabularx}
\usepackage{booktabs} 
\usepackage{enumerate}
\usepackage{caption}
\usepackage[titletoc]{appendix}
\usepackage[numbers]{natbib}
\usepackage{bbm}
\usepackage{url}
\usepackage{subfig}
\usepackage{hyperref}
\usepackage{color}

\catcode`\@=11

\newcommand{\updot}[1]{\raisebox{0.9pt}{$\stackrel{\bullet}{#1}$}} 

\theoremstyle{plain}
\newtheorem{theorem}{Theorem}[section]

\newtheorem{proposition}[theorem]{Proposition}

\theoremstyle{definition}

\theoremstyle{remark}
\newtheorem*{remark}{Remark}

\begin{document}

\title{Managing Information in Queues: The Impact of Giving Delayed Information to Customers}
\author{ 
  Jamol Pender \\ School of Operations Research and Information Engineering \\ Cornell University
\\ 228 Rhodes Hall, Ithaca, NY 14853 \\  jjp274@cornell.edu  \\ 
 \and  
Richard H. Rand \\ Sibley School of Mechanical and Aerospace Engineering \\ Department of Mathematics \\ Cornell University
\\ 535 Malott Hall, Ithaca, NY 14853 \\  rand@math.cornell.edu  \\ 
 \and  
Elizabeth Wesson \\ Department of Mathematics \\ Cornell University
\\ 582 Malott Hall, Ithaca, NY 14853 \\  enw27@cornell.edu  \\ 
 }

\maketitle
\begin{abstract}
Delay or queue length information has the potential to influence the decision of a customer to use a service system.  Thus, it is imperative for service system managers to understand how the information that they provide will affect the performance of the system.  To this end, we construct and analyze two two-dimensional deterministic fluid models that incorporate customer choice behavior based on delayed queue length information.  In the first fluid model, customers join each queue according to a Multinomial Logit Model, however, the queue length information the customer receives is delayed by a constant $\Delta$.  We show that the delay can cause oscillations or asynchronous behavior in the model based on the value of $\Delta$.  In the second model, customers receive information about the queue length through a moving average of the queue length.  Although it has been shown empirically that giving patients moving average information causes oscillations and asynchronous behavior to occur in U.S. hospitals in the work of \citet{dong2015impact}, we analytically and mathematically show for the first time that the moving average fluid model can exhibit oscillations and determine their dependence on the moving average window.  Thus, our analysis provides new insight on how managers of service systems information systems should report queue length information to customers and how delayed information can produce unwanted behavior.     

\end{abstract}


\section{Introduction} \label{sec_intro}

Understanding the impact of providing delay information to customers in service systems is a very important problem in the operations management literature.  Smartphones and internet technology have changed the possibilities for communication between service systems and their potential customers.  Currently, many companies and system managers choose to provide their customers with valuable information that has the potential to influence their choice of using the service.  One example of this communication is delay announcements, which have become important tools of for customers to know how long they will wait on average for someone to start serving them.  These announcements are not only important because they give the customer information about the quality of the service, but also they have the possibility of influencing the possibility that a customer will return to use the service again.  As a consequence, understanding the impact of providing delay or queue length  information to customers on customer choices and system operations, as well as the development of methods to support such announcements, has attracted the attention of the Operations Research and Management communities in the past few years. 

  \begin{figure}
	\centering
		\includegraphics[scale=.75]{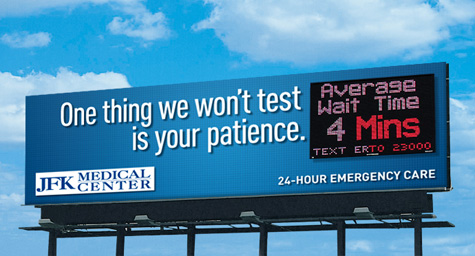}
\caption{Highway Signs Posting Emergency Room Wait Times.} \label{Fig1}
\end{figure}

One example of this new communication between customers and services is in the healthcare industry.  In Figure \ref{Fig1}, we show an example of a typical  billboard sign that many hospitals use for marketing as well as a way of providing information to potential patients.  Since emergency room waiting times can be very long, giving waiting time or queue length information to potential patients is a useful tool for hospitals to communicate to patients when their emergency rooms are relatively underloaded. Much of the current literature that explores the impact of giving customers information about queue lengths and waiting times has been applied in the context of telecommunication systems such as telephone call centers. However, understanding the impact in a healthcare context is much more complicated.  For one, in healthcare, the service discipline is not necessarily first come first serve and can be quite arbitrary.  Given the triage system that is prevalent in hospitals, one could be tempted to model the emergency room with a priority queue.   However, understanding the impact of waiting times and queue lengths on the dynamics in the priority setting is also quite difficult, see for example \citet{pendersampling}.  Moreover, unlike the call center literature where callers are likely to only speak with one agent, the patient experience often involves multiple servers that each have a different purpose in the service process of the patient.  Thus, the patient experience is more like movement through a queueing network.  For example, in a typical emergency room, a patient might interact with a nurse, a doctor, various administrative staff, and even laboratory technicians when tests need to be performed.

Most of the current research on providing queue length or waiting time information to customer focuses on the impact of delay announcements with respect to call centers and telecommunications applications.  There is a vast literature on this subject, which is mostly segmented into three different areas of research.  The first part emphasizes making accurate real-time delay announcements to customers.  In fact, in work by \citet{ibrahim2008real, ibrahim2009real, ibrahim2011real, ibrahim2011wait} develops new estimators for estimating delays in various queueing systems.  They primarily study two types of estimators.  The first estimator is the head of the line (HOL) estimator, which provides the current amount of time that the next customer to get service has waited in line.  The second type of estimator that they study is the delay of the last customer who entered the agent’s service (LES).   These two estimators are different and the papers \citet{whitt1999predicting, ibrahim2008real, ibrahim2009real, ibrahim2011real, ibrahim2011wait} provide a detailed analysis of these estimators.  The second part of the literature addresses how the delay announcements impact the dynamics of the queueing process and how customers respond to the announcements.  The work of \citet{armony2004customer, guo2007analysis, hassin2007information, armony2009impact, guo2009impacts, jouini2009queueing, jouini2011call, allon2011impact, allon2011we, ibrahim2015does, whitt1999improving} and references therein analyzes the impact of delay announcements on the queueing process and the abandonment process of the system.  Finally, the third part of the literature analyzes the customer psychology of waiting.  The work of \citet{hui1996tell, hul1997impact, pruyn1998effects, munichor2007numbers, sarel1998managing, taylor1994waiting} explores the behavioral aspect of customer waiting and how delays affect customer decisions.  This paper is most related to the second area of research; however, it is unique in that it includes customer choice with delay differential equations.  

More recently, there also is work that considers how information can impact queueing systems.  Work by \citet{jennings2015comparisons} compares ticket queues with standard queues.  In a ticket queue, the manager is unaware of when a customer abandons and is only notified of the abandonment when the customer would have entered service.  This artificially inflates the queue length process and the work of \citet{jennings2015comparisons} compares the difference in queue length between the standard and ticket queue.  Follow-up work by \citet{pender2015heavy, pender2015impact} also considers the case when there are dependencies between balking and reneging customers and when the server spends time clearing a customer who has abandoned the system respectively.  However, this work does not consider the aspect of choice and and delays in providing the information to customers, which is the case in many healthcare settings.   

Since hospital networks are more complicated than telephone call centers, understanding the waiting and queueing dynamics is a much harder problem, see for example \citet{armony2015patient}.  Even designing a delay estimator in hospital systems is very difficult and recent work by \citet{plambeck2014forecasting} develops a new emergency department delay estimator that combines methodology from statistical learning theory and queueing theory.  Because of this complexity, it is common that hospitals publish historic average waiting times using a 4-hour moving average and this has been noticed in the work of \citet{dong2015impact}.

\begin{figure}
	\centering
		\includegraphics[scale=.25]{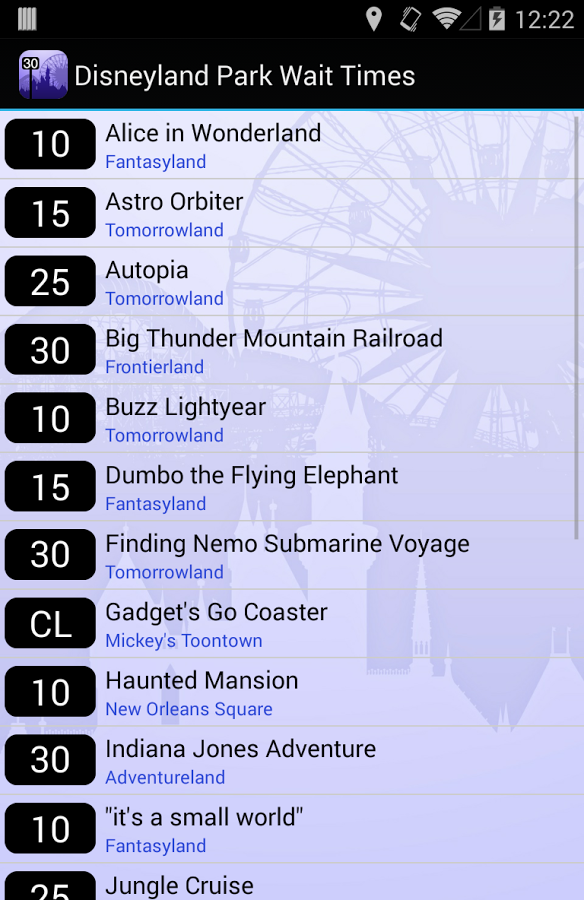}
\caption{Disneyland Park Wait Times App.} \label{Fig2}
\end{figure}

Another useful application of our work is for amusement parks like Disneyland or Six Flags.  In Figure \ref{Fig2}, we show a snapshot of the Disneyland app.  The Disneyland app lists waiting times of various rides in the themepark and customers get to choose which ride that they would want to go on given the waiting times.  However, the wait times on the app are not posted in real-time and are calculated based on moving average of the waiting times.   Thus, our queueing analysis is useful for Disney to syncronize their waiting times for rides across the themepark.  

This paper introduces two new fluid models, which describe the dynamics of customer choice and delay information that customers use to make decisions.  In the first fluid model, the customer receives information about the queue length which is delayed by a parameter $\Delta$.  In the second fluid model, we use a moving average of the queue length over the time interval $\Delta$ to represent the queue length information given to the customer.  The models that we present are useful in two major contexts.  The first context is where the software that communicates with customers is delayed in some fashion, which is common in many hospitals who outsource the computation of their waiting times and queue lengths.  The second context is where the customer reaction to the information is delayed.  This can happen in the Disney example where there is a delay between customers viewing the wait times and them joining the queue.  Thus, the delay does not necessarily need to be a function of the software or a lag in information, it can be caused by the customer behavior and distance from the queue that they are joining.  With these fluid models are able to show that when the delay is small the two queues are balanced and synchronized; however, when the delay is large enough, the  two queue are not balanced and asynchronous. We determine the exact threshold where the the dynamics of the two queues are different for both the constant and moving average models.  Our analysis combines theory from delay differential equations, customer choice models, and stability analysis of differential equations.

 \subsection{Main Contributions of Paper}

The contributions of this work can be summarized as follows:    
\begin{itemize}
\item We develop two new two-dimensional fluid models that incorporate customer choice based on delayed queue length information.  One model uses a constant delay and one model uses a moving average.  
\item We show that the constant delay queueing model can experience oscillations where the two queues are not synchronized and derive the exact threshold where the oscillatory behavior is triggered in terms of the model parameters.  Moreover, we show that the threshold is monotone in terms of the arrival rate.  
\item We show that the moving average queueing model can experience oscillations where the two queues are not synchronized.  However, unlike the constant delay system, the threshold is not monotone as a function of the arrival rate.
\end{itemize} 


\subsection{Organization of Paper}

The remainder of this paper is organized as follows. Section~\ref{sec_CD} describes a constant delay fluid model.  We derive the critical delay threshold under which the queues are balanced if the delay is below the threshold and the queues are asynchronized if the delay is above the threshold.  We also show that the instability is preserved as long as the delay is increased.  Section~\ref{sec_MA} describes a constant moving average delay fluid model.  We derive the critical delay threshold under which the queues are balanced if the delay is below the threshold and the queues are asynchronized if the delay is above the threshold.  We also show that the instability is preserved as long as the delay is increased in certain regions of the parameter space. Finally in Section~\ref{sec_conclusion}, we conclude with directions for future research related to this work.


\section{Constant Delay Fluid Model }\label{sec_CD}

In this section, we present a new fluid model with customer choice based on the queue length with a constant delay.      
Thus, we begin with two infinite-server queues operating in parallel, where customers choose which queue to join by taking the size of the queue length into account. However, we add the twist that the queue length information that is reported to the customer is delayed by a constant $\Delta$.  Therefore, the queue length that the customer receives is actually the queue length $\Delta$ time units in the past.  An example of this delay is given in Figure \ref{Fig:Jfk}, which is JFK Medical Center in Boynton Beach, Florida.  In Figure \ref{Fig:Jfk}, the average wait time is reported to be 12 minutes.  However, in the top right of the figure we see that the time of the snapshot was 4:04pm while the time of a 12 minute wait is as of 3:44pm.  Thus, there is a delay of 20 minutes in the reporting of the wait times in the emergency room and this can have an important impact on the system dynamics as we will show in the sequel.  

The choice model that we use to model these dynamics is identical to that of a Multinomial Logit Model (MNL) where the utility for being served in the $i^{th}$ queue with delayed queue length  $Q_i(t-\Delta)$ is $u_i(Q_i(t-\Delta)) =  − Q_i(t-\Delta)$. Thus, in a stochastic context with two queues, the probability of going to the first queue is given by the following expression

\begin{eqnarray}
p_1( Q_1(t), Q_2(t), \Delta ) &=& \frac{\exp(-Q_1(t-\Delta))}{\exp(-Q_1(t-\Delta)) + \exp(-Q_2(t-\Delta))} 
\end{eqnarray}
and the probability of going to the second queue is 
\begin{eqnarray}
p_2( Q_1(t), Q_2(t), \Delta ) &=& \frac{\exp(-Q_2(t-\Delta))}{\exp(-Q_1(t-\Delta)) + \exp(-Q_2(t-\Delta))} .
\end{eqnarray}

\begin{figure}
	\centering
		\includegraphics[scale=.25]{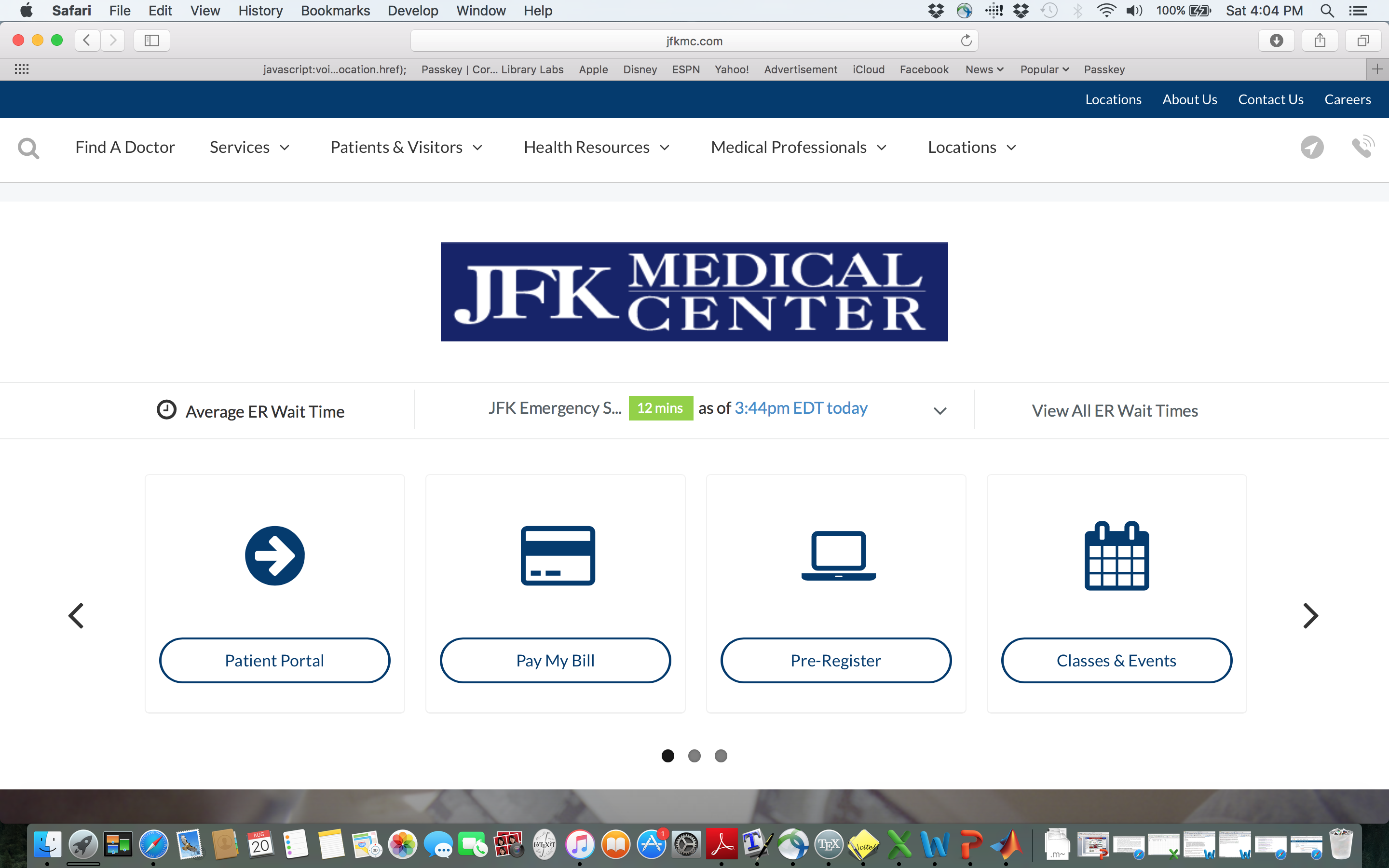}
\caption{JFK Medical Center Online Reporting.} \label{Fig:Jfk}
\end{figure}

Since the main goal of our analysis is to provide insight into the dynamics of the system when delayed information is given to customers, we analyze a fluid model of the system instead of the actual stochastic process, which is more difficult.  Moreover, the fluid model enables us to understand the mean dynamics of the system when the number of arrivals in the system is large, which is the case in themeparks like Disneyland.  However, since we analyze the fluid model instead of the real stochastic system, we no longer have probabilities in our choice model.  Instead, we now have rates at which customers join each of the two queues.  In our fluid model, we assume that the sum of the arrival rate of customers to both queues is equal to the constant rate $\lambda$.   Thus, customers join the first queue, $q_1(t)$, at rate 

\begin{equation}
\lambda \cdot \frac{\exp(-q_1(t-\Delta))}{\exp(-q_1(t-\Delta)) + \exp(-q_2(t-\Delta))}
\end{equation}
and therefore patients join the second queue, $q_2(t)$, at rate
\begin{equation}
\lambda \cdot \frac{\exp(-q_2(t-\Delta))}{\exp(-q_1(t-\Delta)) + \exp(-q_2(t-\Delta))} .
\end{equation} 
Thus, our model for customer choice infinite server queues with delayed information can be represented by the two dimensional system of delay differential equations
\begin{eqnarray}
\updot{q}_1(t) &=& \lambda \cdot \frac{\exp(-q_1(t-\Delta))}{\exp(-q_1(t-\Delta)) + \exp(-q_2(t-\Delta))} - \mu q_1(t) \label{ddecd1} \\
\updot{q}_2(t) &=& \lambda \cdot \frac{\exp(-q_2(t-\Delta))}{\exp(-q_1(t-\Delta)) + \exp(-q_2(t-\Delta))} - \mu q_2(t) \label{ddecd2}
\end{eqnarray}
where we assume that $q_1(t)$ and $q_2(t)$ start with different initial functions $\varphi_1(t)$ and $\varphi_2(t)$ on the interval $[-\Delta,0]$.  

\begin{remark}
When the two delay differential equations are started with the same initial functions, they are identical for all time because of the symmetry of the problem.  Therefore, we will start the system with non-identical initial conditions so the problem is no longer trivial and the dynamics are not identical.  
\end{remark}


\begin{figure}
\captionsetup{justification=centering}
		\hspace{-.35in}~\includegraphics[scale = .22]{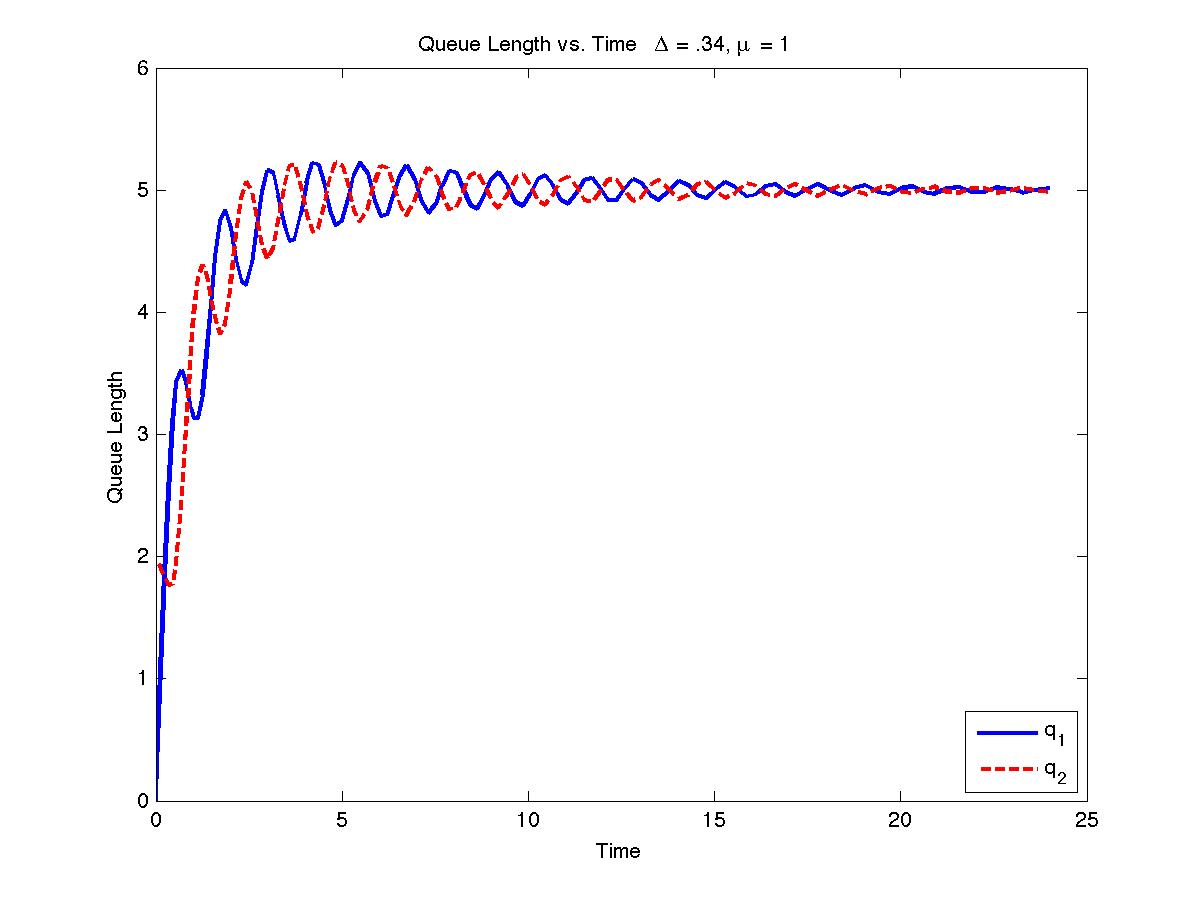}~\hspace{-.3in}~\includegraphics[scale = .22]{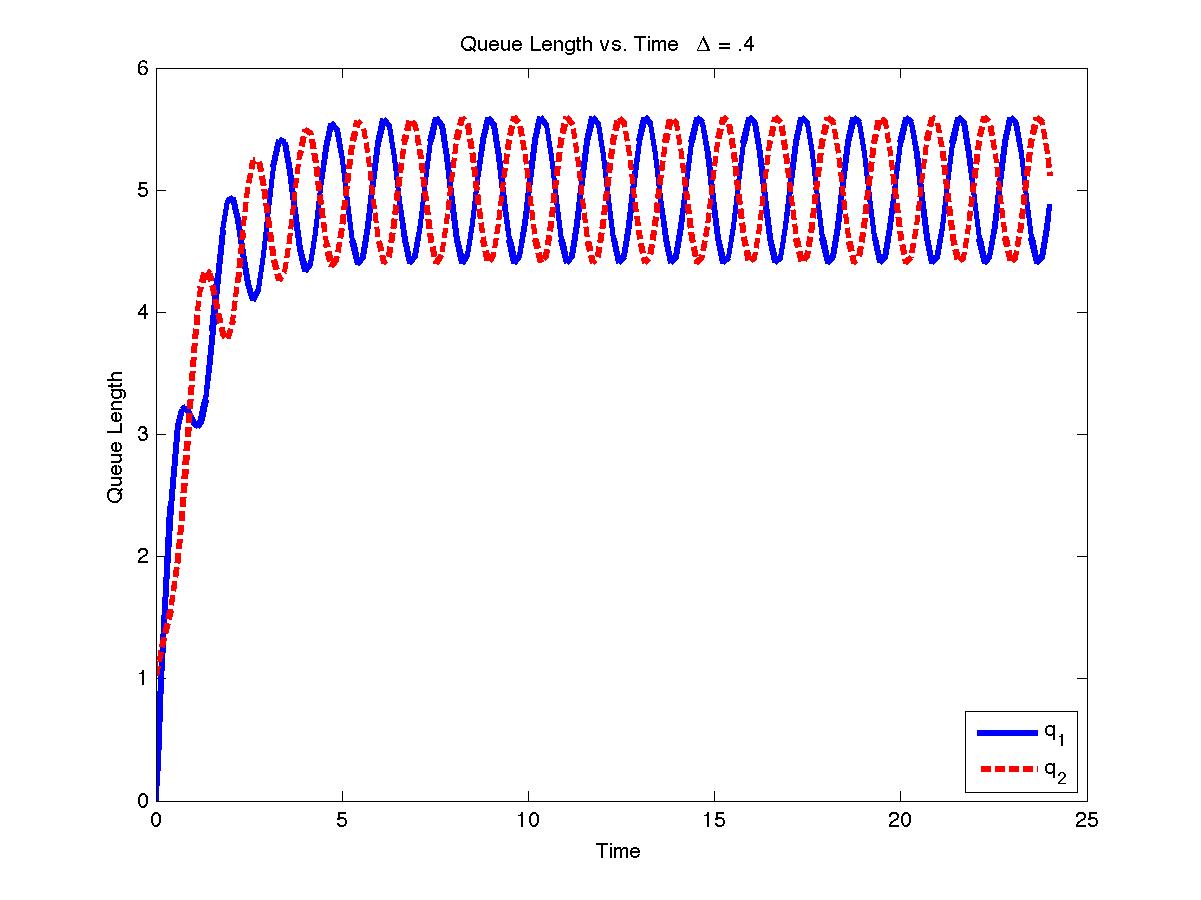}
\caption{Numerical integration of fluid model for $\lambda=10$, $\mu=1$, $\Delta_{cr} = .3614$ \\ $\Delta = .34$  (Left)\quad  $\Delta = .4$ (Right).} \label{Fig3}
\end{figure}

In Figure \ref{Fig3}, we see qualitatively different behavior of the fluid model equations when $\Delta = .34$ and $\Delta = .4$.  It is clear from the left plot in Figure \ref{Fig3} that the two queues are synchronized and converge to the $\Delta = 0$ equilibrium solution.  However, on the right plot of Figure \ref{Fig3}, we see that the two queues are not synchronized and exhibit oscillatory and asynchronous behavior.    It turns out that for the model parameters presented in Figure \ref{Fig3}, asynchronous dynamics will not occur as in right plot of Figure \ref{Fig3} if the delay $\Delta < .03614$.  Otherwise, oscillations and asynchronous dynamics will exist for both queues.  However, this change in behavior can be explained by the fact that the equilibrium points of the queue length delay differential equations transition from stable to unstable where a limit cycle is born.  This situation is known as a Hopf bifurcation  and will be explained in more detail later in the paper.  In the next theorem, we show how to derive the critical delay that serves as the boundary of oscillatory and non-oscillatory dynamics of the system. 

\begin{theorem}
For the constant delay choice model, the critical delay parameter is given by the following expression
\begin{equation}
\Delta_{cr}(\lambda, \mu)=\frac{2 \arccos(-2\mu/\lambda)}{ \sqrt{\lambda^2-4\mu^2}}.
\end{equation}

\begin{proof}
We split the proof into several parts to help readers understand the important ingredients that are necessary to prove the theorem; we outline the main parts of the proof in bold text.   
 
\paragraph{Computing the Equilibrium}
 
The first part of the proof is to compute an equilibrium for the solution to the delay differential equations.  In our case, the delay differential equations given in Equations \ref{ddecd1} - \ref{ddecd2} are symmetric.  Moreover, in the case where there is no delay, the two equations converge to the same point since in equilibrium each queue will receive exactly one half of the arrivals and the two service rates are identical.  This is also true in the case where the arrival process contains delays in the queue length since in equilibrium, the delayed queue length is equal to the non-delayed queue length.  Thus, we have in equilibrium that 
\begin{equation}
q_1(t-\Delta) = q_2(t-\Delta) = q_1(t) = q_2(t) = \frac{\lambda}{2 \mu} \quad \mathrm{ as \ } t \to \infty.
\end{equation}

\paragraph{Understanding the stability of the equilibrium}  Now that we know the equilibrium for Equations \ref{ddecd1} - \ref{ddecd2}, we need to understand the stability of the delay differential equations around the equilibrium.  The first step in doing this is to set each of the queue lengths to the equilibrium points plus a perturbation.  Thus, we set each of the queue lengths to 
\begin{eqnarray}
q_1(t) &=& \frac{\lambda}{2 \mu} + u_1(t) \label{sub1}\\
q_2(t) &=& \frac{\lambda}{2 \mu} + u_2(t) \label{sub2}
\end{eqnarray}
where $u_1(t)$ and $u_2(t)$ are pertubations about the equilibrium point $ \frac{\lambda}{2 \mu}$.  By substituting Equations \ref{sub1} - \ref{sub2} into Equations \ref{ddecd1} - \ref{ddecd2} respectively and linearizing around the point $u_1(t) = u_2(t) = 0$, we have that the perturbations solve the following delay differential equations 

\begin{eqnarray}
\updot{u}_1(t) &=& -\frac{\lambda}{4} \cdot (u_1(t-\Delta)-u_2(t-\Delta)) - \mu \cdot  u_1(t) \label{pert1}\\
\updot{u}_2(t) &=& -\frac{\lambda}{4} \cdot (u_2(t-\Delta)-u_1(t-\Delta)) - \mu \cdot u_2(t) \label{pert2}.
\end{eqnarray}

\paragraph{Uncoupling the differential equations}  In their current form the delay differential equations for the perturbations do not yield any insight since they are coupled together.  However, we can make a simple transformation and the resulting delay differential equations will become uncoupled.  Thus, we apply the following transformation to uncouple the system of equations:
\begin{eqnarray}
v_1(t) &=& u_1(t) + u_2(t)   \label{subuncup1} \\
v_2(t) &=& u_1(t) - u_2(t)  \label{subuncup2} .
\end{eqnarray}
This transformation yields the following delay equations for the transformed perturbations $v_1(t)$ and $v_2(t)$
\begin{eqnarray}
\updot{v}_1(t) &=& -\mu \cdot  v_1(t) \label{uncup1} \\
\updot{v}_2(t) &=& -\frac{\lambda}{2} \cdot v_2(t-\Delta) - \mu \cdot v_2  \label{uncup2} .
\end{eqnarray}
Since Equation \ref{uncup1} is linear and does not depend on the delay parameter $\Delta$, we can explicitly solve for the solution.  The general solution of Equation \ref{uncup1}  is $v_1(t) = c_1 \exp({-\mu t})$ and is bounded and stable.  Thus, it remains for us to analyze Equation \ref{uncup2}.  To do this, we substitute the following expression for $v_2(t)$
\begin{equation}
v_2(t) = \exp(r t).
   \end{equation}    
This substitution now yields a transcendental equation for the parameter $r$ and is given by the equation
\begin{equation}\label{r-trans}
r = -\frac{\lambda}{2} \cdot \exp(-r \Delta) - \mu.
\end{equation}
Now it remains to find the transition between stable and unstable solutions.  When the parameter $r$ crosses the imaginary axis, the stability of the equilibrium changes. In the full nonlinear system given in Equations \ref{ddecd1} - \ref{ddecd2}, this transition generally occurs in a Hopf bifurcation, in which a pair of roots crosses the imaginary axis and a limit cycle is born. To find the critical value of $\Delta$ for the change of stability, we set $r=i\omega$, which yields the following equation
\begin{equation} \label{iw}
i\omega = -\frac{\lambda}{2}(\cos\omega \Delta- i \sin\omega \Delta)-\mu.
\end{equation}
Writing the real and imaginary parts of  Equation \ref{iw}, we have that:
\begin{equation}\label{realeqn}
0 = -\frac{\lambda}{2}\cos\omega \Delta-\mu
\end{equation}
for the real part and 
\begin{equation}\label{imageqn}
\omega = \frac{\lambda}{2}\sin\omega \Delta
\end{equation}
for the imaginary part.  
Solving Equations \ref{realeqn} - \ref{imageqn} for the functions $\sin\omega \Delta$ and $\cos\omega \Delta$ we get that
\begin{equation}
\cos\omega \Delta = -\frac{ 2\cdot \mu}{\lambda}
\end{equation}
for the real part and 
\begin{equation}
\sin\omega \Delta = \frac{ 2\cdot \omega}{\lambda}
\end{equation}
for the imaginary part.  Now by squaring both equations and adding them together we get that 
\begin{equation}
\lambda^2 =  4 \cdot (\omega^2 + \mu^2),
\end{equation}
which by some rearranging yields
\begin{equation}\label{omeg}
\omega = \frac{1}{2}\sqrt{\lambda^2-4\mu^2}	.
\end{equation}
Now if we go back to use Equation \ref{realeqn} to find an expression for the critical delay $\Delta_{cr}$.  From Equation \ref{realeqn} we know that 
\begin{equation}
\cos\omega \Delta = -\frac{ 2\cdot \mu}{\lambda}
\end{equation} 
and therefore by taking the arcosine of both sides, we have that
\begin{eqnarray}
 \Delta = \frac{\arccos(-2\mu/\lambda)}{\omega}
\end{eqnarray}
 Finally substituting our expression for $\omega$ in Equation \ref{omeg}, we are able to obtain the final expression for the critical delay:
 \begin{equation}\label{finalcd}
 \fbox{ $\displaystyle{  \Delta_{cr} = \frac{2 \arccos(-2\mu/\lambda)}{ \sqrt{\lambda^2-4\mu^2}}   }  $}
\end{equation}
\end{proof}
\end{theorem}

\begin{figure}
\captionsetup{justification=centering}
		\centering \includegraphics[scale=.22]{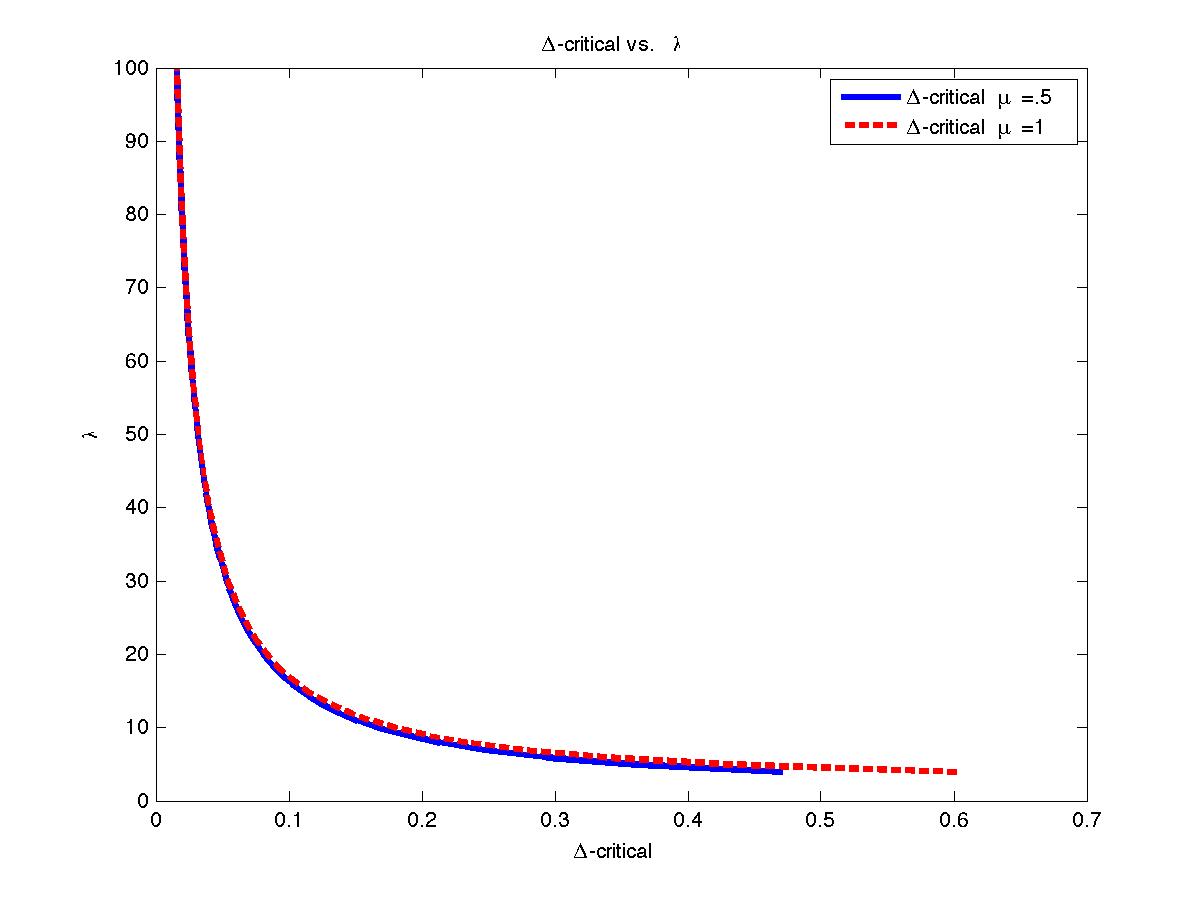}~\hspace{-.3in}~\includegraphics[scale=.22]{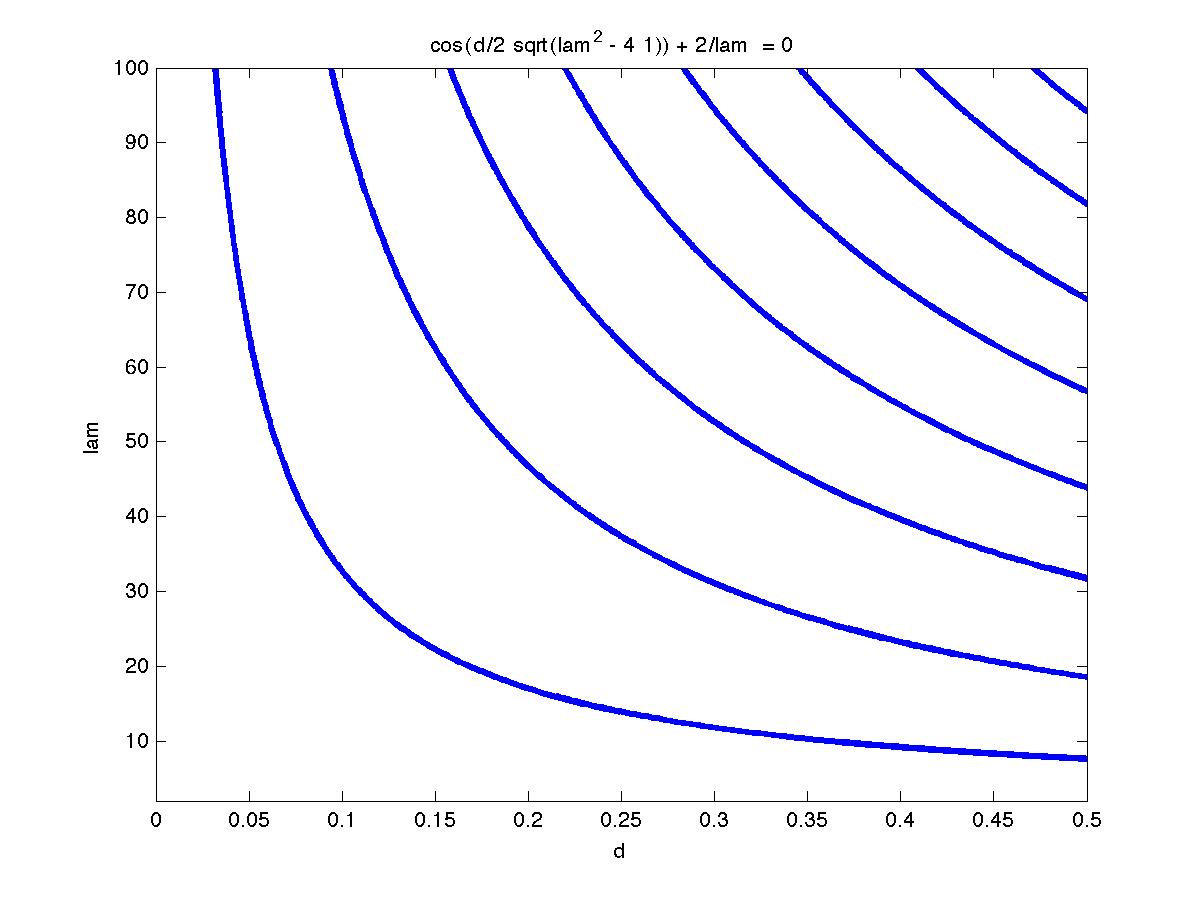}
\caption{$\Delta_{cr}$ as a function of $\lambda$ when $\mu=.5$ and  $\mu = 1$ (Left) \\ EZ-Plot of $\Delta_{cr}$ as a function of $\lambda$ when $\mu = 1$ (Right).} \label{Fig4}
\end{figure}

\begin{figure}
		\hspace{-.35in}~\includegraphics[scale=.22]{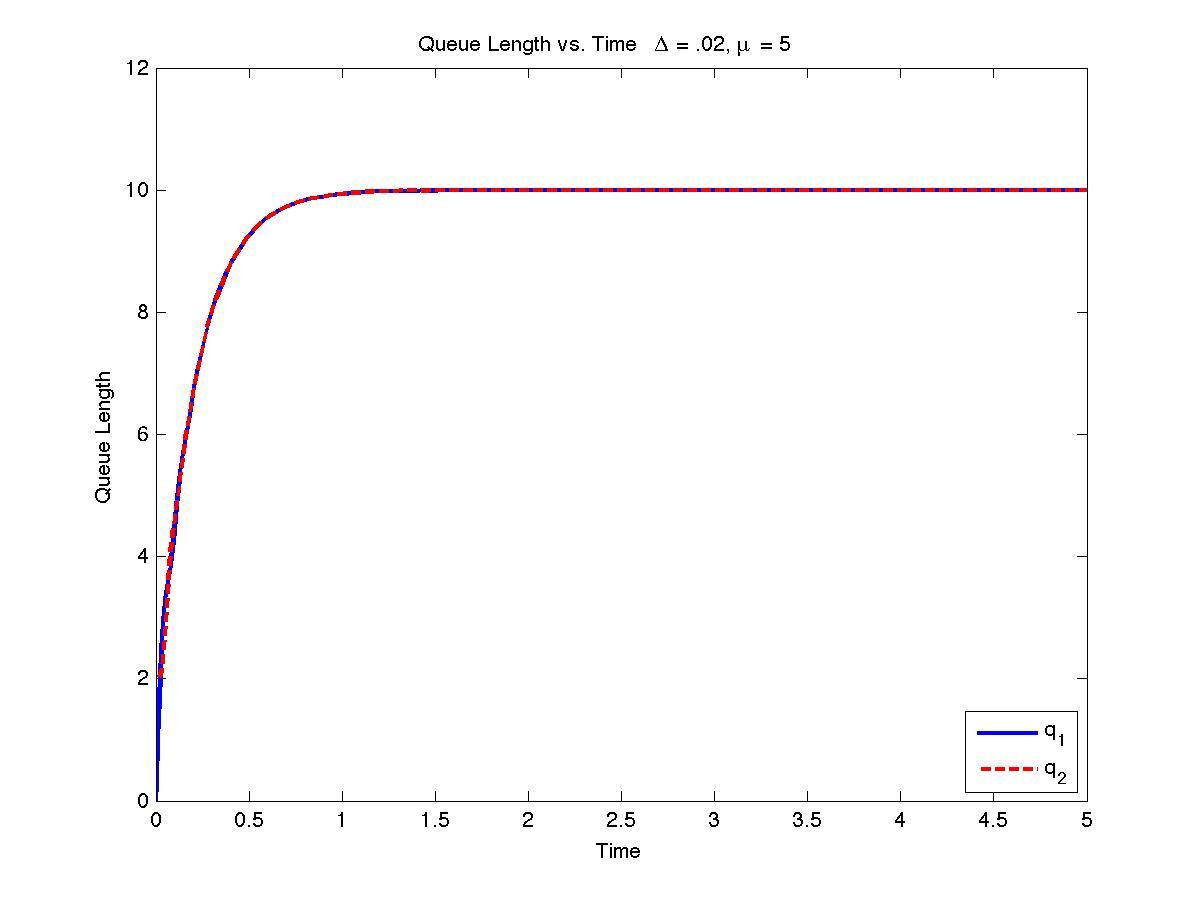}~\hspace{-.3in}~\includegraphics[scale=.22]{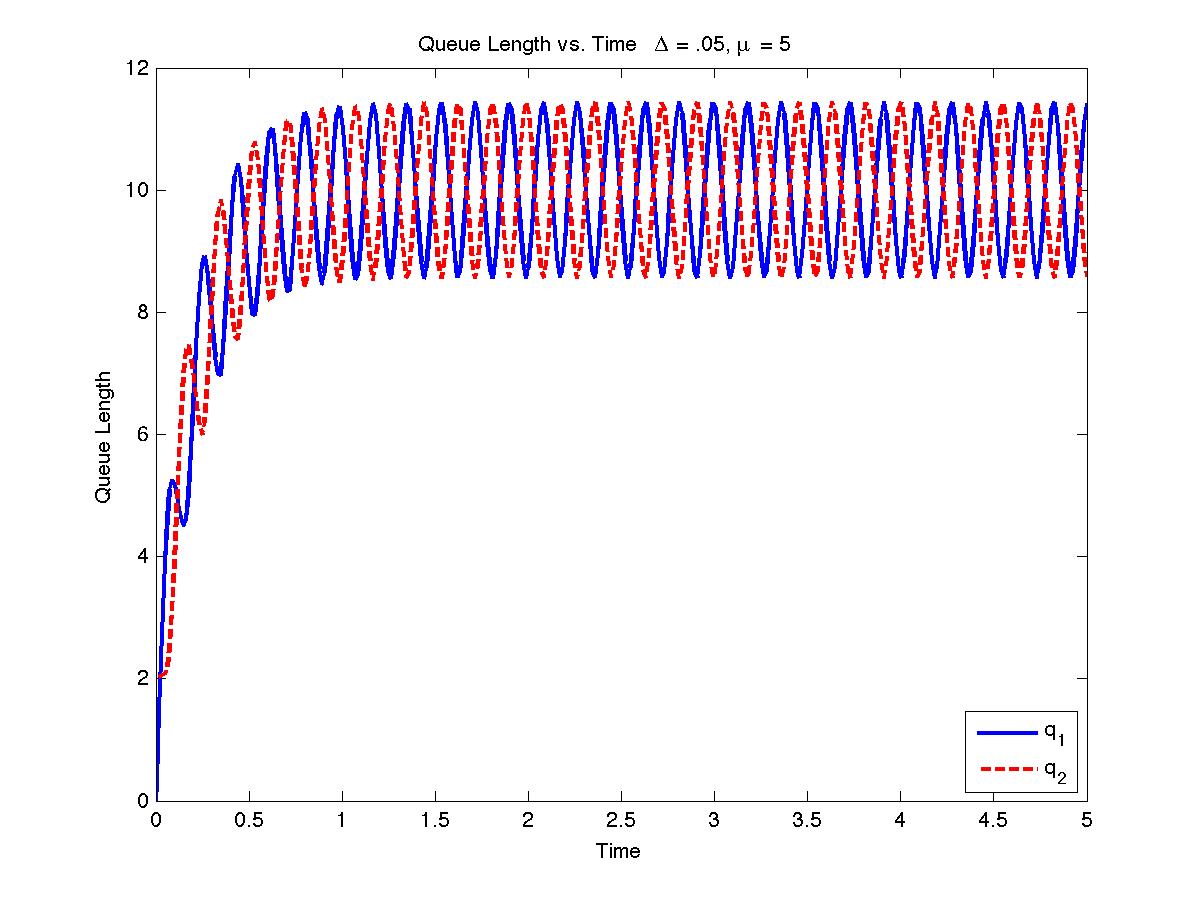}
\caption{Numerical integration of fluid model for $\lambda=100$, $\mu=5$, $\Delta = .02$ (left) $\Delta = .05$ (right).} \label{Fig5}
\end{figure}

In Figure \ref{Fig4}, we show the dependence of the critical delay $\Delta_{cr}$ as a function of the model parameter $\lambda$ while keeping $\mu$ constant.  On the left of Figure \ref{Fig4}, we plot Hopf curves when $\mu=.5$ and $\mu=1$.  We see that the curves are very similar for both values of $\mu$ and that the critical delay value decreases as $\lambda$ increases.  Moreover, on the right of Figure \ref{Fig4}, we use the function in Matlab called \textbf{EZ-Plot}, to plot the various critical delay values.  We see that there are several Hopf bifurcation curves, which could indicate that there could possibly be several regions where the stability of the delay differential equation system might change i.e the example of Figure \ref{Fig3} could be reversed.  However, we will show in the sequel that this reversal of stability is impossible since the all of the roots pass from the left half-plane to right half-plane, which prevents the system from becoming stable again.


\subsection{Hopf Bifurcation Curves in the Constant Delay Model}

On the right of Figure \ref{Fig4}, as the first Hopf curve is crossed, we see a stable limit cycle born. However, after the next Hopf curve is crossed (as parameters are slowly changed), there are various possibilities:

\begin{enumerate}
\item The pair of roots which crossed into the right half-plane in the first Hopf, may cross back into the left half-plane. Thus all roots are in the left half-plane again and the equilibrium reverts to stability. The limit cycle which was born in the first Hopf bifurcation shrinks to nothing and disappears.
\item Another pair of roots may cross into the right half-plane, so that two pairs of roots are now in the left half-plane. A new limit cycle may be born, but the stability of the equilibrium does not change. The new limit cycle is expected to be unstable.
\item One of many possible degenerate cases: multiple pairs may cross at once, or the first pair of roots may re-cross simultaneously with the next pair crossing, or a pair of roots may touch the imaginary axis but not cross. 
\end{enumerate}

Numerical integration with the Matlab delay differential equation package \textbf{dde23} showed that there was only one stable limit cycle observed, no matter how many of the Hopf curves are crossed. However, this does not tell us what happens to the various pairs of imaginary roots which occur on the various Hopf curves. Do they all pass from left half-plane to right half-plane, or do some of them come back in the opposite direction?  The purpose of the remainder of this subsection is to address this issue.  

Now suppose that the delay $\Delta$ is close to a critical value for a Hopf bifurcation.  We then make a slight perturbation from the critical value for the Hopf bifurcation i.e
\begin{equation}\label{Dperturb}
\Delta = \Delta_0 + \epsilon \Delta_1
\end{equation}
where $\epsilon \ll 1$.  Then the root $r$ will be slightly perturbed from the pure imaginary value it would take at $\Delta = \Delta_0$.  
That is, we can write
\begin{equation}\label{rperturb}
r = i \omega + \epsilon ( i r_1 + r_2)
\end{equation}
where $r_1$ and $r_2$, the imaginary and real parts of the perturbation, may be determined in terms of $\Delta_0$ and $\Delta_1$.

\begin{proposition}\label{prop1}
Suppose that we make a slight perturbation on the order of $\epsilon \Delta_1$ near a critical delay value for a Hopf bifurcation, then the real part of r is equal to 
\begin{equation}\label{rcd2}
r_2 = \frac{4 \omega^2 \Delta_1}{8 \Delta_0 \mu + \Delta_0^2 \lambda^2 + 4}.
\end{equation}
In particular, $r_2$ has the same sign as $\Delta_1$.  
\end{proposition}

\begin{proof}
When $\epsilon = 0$, we reduce back to the original critical threshold of Equation \ref{finalcd}.  Now we substitute Equations \ref{Dperturb} and \ref{rperturb} into Equation 
\ref{r-trans} and do a Taylor expansion for small values of $\epsilon$. Collecting real and imaginary parts, Equation \ref{r-trans} becomes
\begin{align}
 0 = & \, i \left(\omega -\frac{1}{2} \lambda  \sin \left(\Delta _0 \omega \right)\right)+\frac{1}{2} \lambda  \cos \left(\Delta
   _0 \omega \right)+\mu \nonumber \\
   & + \epsilon \left[ i \left( r_1-\frac{1}{2} \lambda  \left(\Delta _0 r_2 \sin \left(\Delta _0 \omega \right)+\cos \left(\Delta _0 \omega \right)
   \left(\Delta _0 r_1-\Delta _1 \omega \right)\right) \right) \right. \nonumber \\
   & \left. + r_2-\frac{1}{2} \lambda  \left(\sin \left(\Delta _0 \omega \right) \left(\Delta _1 \omega +\Delta _0 r_1\right)+\Delta
   _0 r_2 \cos \left(\Delta _0 \omega \right)\right) \right] \nonumber \\
   & + O(\epsilon^2)
\end{align}

We set the real and imaginary parts of the $O(\epsilon)$ term separately equal to 0, and solve for $r_1$ and $r_2$ since there are two linear equations and two unknowns.  Solving for the real part of $r$, we find that $r_2$ is equal to the following value  
\begin{equation}\label{rcd}
r_2 = \frac{- 2 \Delta_1 \lambda \omega \sin \omega \Delta_0}{4 \Delta_0 \lambda \cos \omega \Delta_0 - \Delta_0^2 \lambda^2 - 4}.
\end{equation}

Now if we substitute the expressions for $\sin \omega \Delta_0$ and $\cos \omega \Delta_0$ from Equations \ref{realeqn} - \ref{imageqn}, then we have that 

\begin{equation}\label{rcd2}
r_2 = \frac{4 \omega^2 \Delta_1}{8 \Delta_0 \mu + \Delta_0^2 \lambda^2 + 4}.
\end{equation}
\end{proof}

Thus $r_2$ has the same sign as $\Delta_1$. This means that as $\Delta$ increases past $\Delta_0$, the critical Hopf value, $r$ crosses the imaginary axis from left to right. This analysis holds for every Hopf curve, which implies that the initial instability that occurs after passing through the first Hopf curve is preserved and does not change as we pass through more Hopf curves.  In Figure \ref{Fig5}, we provide an additional numerical example to illustrate the change in stability before and after our critical delay $\Delta_{cr}$.  Once again, we see that for all values of the delay before the critical delay threshold, the two queues synchronize, balance is achieved, and the system is stable near the equilibrium point.  However, for all values after the critical delay threshold, the two queues exhibit asynchronous behavior and are not stable near the equilibrium point.  However, not all real systems use a constant delay to report to their customers about the queue length or waiting time.  It has been observed in \citet{dong2015impact} that some service systems such as hospitals use a moving average.  Thus, a moving average fluid model will be analyzed in the subsequent section.


\section{Moving Average Delay Fluid Model} \label{sec_MA}

In this section, we present another fluid model with customer choice and where the delay information presented to the customer is a moving average. This model assumes that customers are informed about the queue length, but in the form of a moving average of the queue length between the current time and $\Delta$ time units in the past.  Like in the previous model, customers also have the choice to join two parallel infinite server queues and they join according to the same multinomial logit model.   We also assume that the total rate at which customers show up to the system is given by the parameter $\lambda$ which is a constant.  However, unlike the previous model, we assume that customers join the first queue at rate 
\begin{equation}
\lambda \cdot \frac{\exp\left(- \frac{1}{\Delta} \int^{t}_{t-\Delta} q_1(s) ds \right)}{\exp\left(- \frac{1}{\Delta} \int^{t}_{t-\Delta} q_1(s) ds \right) + \exp\left(- \frac{1}{\Delta} \int^{t}_{t-\Delta} q_2(s) ds \right)}
\end{equation}
and join the second queue at rate
\begin{equation}
\lambda \cdot \frac{\exp\left(- \frac{1}{\Delta} \int^{t}_{t-\Delta} q_2(s) ds \right)}{\exp\left(- \frac{1}{\Delta} \int^{t}_{t-\Delta} q_1(s) ds \right) + \exp\left( - \frac{1}{\Delta} \int^{t}_{t-\Delta} q_2(s) ds \right)} .
\end{equation} 
Thus, our model for customer choice with delayed information in the form of a moving average can be represented by a two dimensional system of functional differential equations
\begin{eqnarray}
\updot{q}_1(t) &=& \lambda \cdot \frac{\exp\left(- \frac{1}{\Delta} \int^{t}_{t-\Delta} q_1(s) ds \right)}{\exp\left(- \frac{1}{\Delta} \int^{t}_{t-\Delta} q_1(s) ds \right) + \exp\left(- \frac{1}{\Delta} \int^{t}_{t-\Delta} q_2(s) ds \right)} - \mu q_1(t) \\
\updot{q}_2(t) &=& \lambda \cdot \frac{\exp\left(- \frac{1}{\Delta} \int^{t}_{t-\Delta} q_2(s) ds \right)}{\exp\left(- \frac{1}{\Delta} \int^{t}_{t-\Delta} q_1(s) ds \right) + \exp\left(- \frac{1}{\Delta} \int^{t}_{t-\Delta} q_2(s) ds \right)}- \mu q_2(t)
\end{eqnarray}
where we assume that $q_1$ and $q_2$ start at different initial functions $\varphi_1(t)$ and $\varphi_2(t)$ on the interval $[-\Delta,0]$.  

\begin{remark}
We should also mention that if we initialize the differential equations with the same initial conditions, then the moving average delay differential equations are identical for all time because of the symmetry of the problem.  Starting the two queues with identical initial conditions places both queues on an invariant manifold from which it cannot escape. Therefore, we start the system with non-identical initial conditions so the problem is no longer trivial and the two queues start off the invariant manifold.  
\end{remark}

On the onset this problem is seemingly more difficult than the constant delay setting since the ratio now depends on a moving average of the queue length during a delay period $\Delta$.  To simplify the notation, we find it useful to define the moving average of the $i^{th}$ queue over the time interval $[t - \Delta, t]$ as
\begin{eqnarray}\label{mai}
 m_i(t,\Delta) = \frac{1}{\Delta} \int^{t}_{t-\Delta} q_i(s) ds.
\end{eqnarray}

A key observation to make is that the moving average itself solves a delay differential equation.  In fact, by differentiating Equation \ref{mai} with respect to time, it can be shown that the moving average of the $i^{th}$ queue is the solution to the following delay differential equation
\begin{eqnarray} 
\updot{m}_i(t,\Delta) =   \frac{1}{\Delta} \cdot \left(  q_i(t) - q_i(t - \Delta) \right), \quad i \in \{1,2\}.
\end{eqnarray}

Leveraging the above delay equation for the moving average, we can describe our moving average fluid model with the following four dimensional system of delay differential equations

\begin{eqnarray}
\updot{q}_1 &=& \lambda \cdot \frac{\exp\left(- m_1(t) \right) }{\exp\left(- m_1(t) \right) + \exp\left(- m_2(t) \right)} - \mu q_1(t)    \label{ma1}  \\
\updot{q}_2 &=& \lambda \cdot \frac{\exp\left(- m_2(t) \right)}{\exp\left(-m_1(t) \right) + \exp\left(- m_2(t) \right)}- \mu q_2(t) \\
\updot{m}_1 &=&  \frac{1}{\Delta} \cdot \left(  q_1(t) - q_1(t - \Delta) \right) \\
\updot{m}_2 &=&  \frac{1}{\Delta} \cdot \left( q_2(t) - q_2(t - \Delta) \right) \label{ma2}.
\end{eqnarray}
%
%

\begin{figure}
\captionsetup{justification=centering}
		\hspace{-.35in}~\includegraphics[scale=.22]{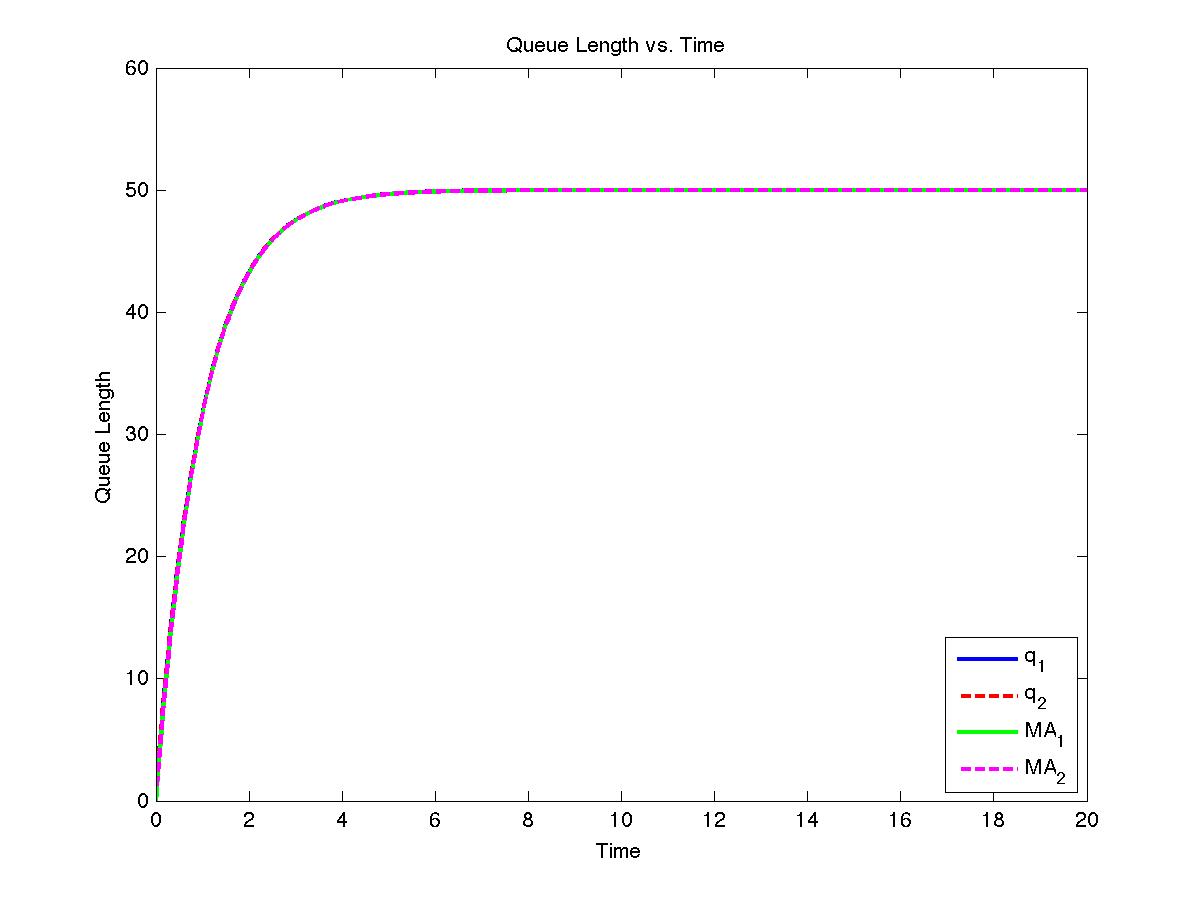}~\hspace{-.3in}~\includegraphics[scale=.22]{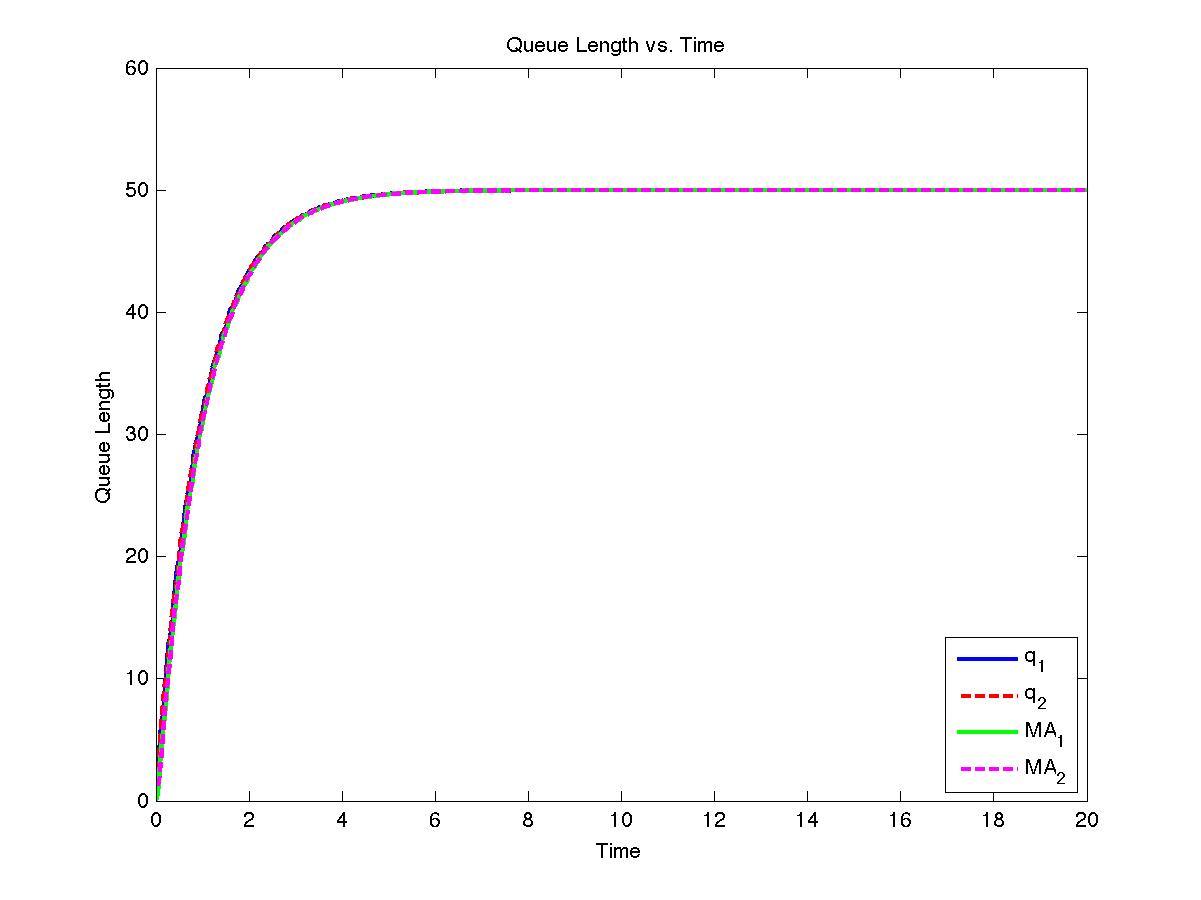} 
	\caption{ $\lambda=100$, $\mu=1$,  \\ $\Delta = .02$  (Left) \quad  $\Delta = .1$ (Right).}\label{Fig6}
\end{figure}

In Figure \ref{Fig6}, we plot two examples of the moving average delay differential equations.  In the example on the left of Figure \ref{Fig6}, the differential equations converge to the equilibrium of $\lambda/(2\mu)$ when $\Delta = .02$ and the dynamics are stable.  It also seems like the plot on the right of Figure \ref{Fig6} also converges to the equilibrium and is also stable.  However, it is not stable and requires even closer observation.  In Figure \ref{Fig7}, we zoom in and look at the dynamics of each example.  On the left of Figure \ref{Fig7}, a closer look reveals that the differential equations are indeed stable.  However, on the right when $\Delta =.1$, it is observed that the differential equations are not stable.  Although the two queues appear to be stable in Figure \ref{Fig6}, the amplitude is too small to detect the asynchronous dynamics of the two queues. Thus, like in the previous fluid model, we need to understand the dynamics of the fluid model near the equilibrium to determine when the two queues will exhibit asynchronous behavior.  The next theorem provides insight for understanding when the equilibrium behavior will be stable or unstable.  

\begin{figure}
\captionsetup{justification=centering}
		\hspace{-.35in}~\includegraphics[scale=.47]{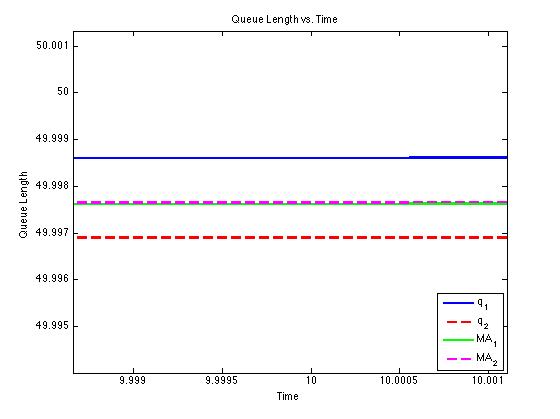}~\hspace{-.3in}~\includegraphics[scale=.47]{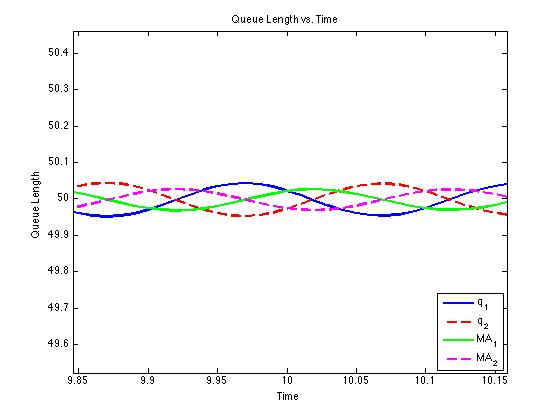} 
	\caption{ Zoomed in.  $\lambda=100$, $\mu=1$,  \\ $\Delta = .02$  (Left) \quad  $\Delta = .1$ (Right).}\label{Fig7}
\end{figure}

In Figure \ref{Fig8}, we plot the Hopf curves for the moving average model when the $\mu =1$.  These curves are different from what would be plotted in the Matlab function \textbf{ez-plot}.  One reason is that we square the cosine and sine functions, which introduces extraneous roots that do not exist.  Thus, the plot given in Figure \ref{Fig8} excludes these extraneous roots.  Moreover, unlike the constant delay case, we also see a linear curve at the bottom of Figure \ref{Fig8}.  This line represents where $\omega =0$ and is another root of the equation.  However, the stability is unchanged on this line, therefore it is not interesting to study.   

\begin{figure}
\captionsetup{justification=centering}
\begin{center}
		\includegraphics[scale=.4]{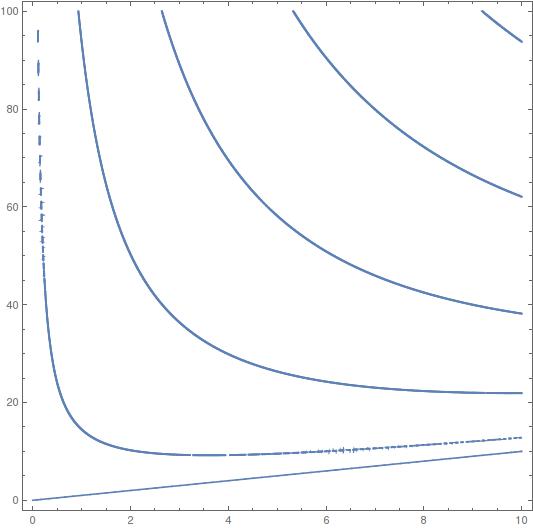}
		\end{center}
	\caption{Mathematica-Plot of Hopf Curves for Moving Average Model ($\mu =1$) \\  }\label{Fig8}
\end{figure}

\begin{theorem}
For the moving average fluid model given by Equations \ref{ma1} - \ref{ma2}, the critical delay parameter is the solution to the following transcendental equation
\begin{equation}
\sin\left( \Delta \cdot \sqrt{\frac{\lambda}{\Delta} - \mu^2} \right) + \frac{2 \mu \Delta}{\lambda} \cdot \sqrt{\frac{\lambda}{\Delta} - \mu^2} = 0 .
\end{equation}

\begin{proof}
Like in the constant delay setting, we will split the proof into several parts to help readers understand the important ingredients that are necessary to prove the theorem. 

\paragraph{Computing the Equilibrium} The first part of the proof is to compute an equilibrium for the solution to the delay differential equations.  In our case, the delay differential equations given in Equations \ref{ma1} - \ref{ma2} are symmetric.  Moreover, in the case where there is no delay, the two equations converge to the same point since in equilibrium each queue will receive exactly one half of the arrivals and the two service rates are identical.  This is also true in the case where the arrival process contains delays in the queue length since in equilibrium, the delayed queue length is equal to the non-delayed queue length.  It can be shown that there is only one equilibrium where all of the states are equal to each other.  One can prove this by substituting $q_2 = \lambda/\mu - q_1$ in the steady state verison of Equation \ref{ma1} and solving for $q_1$.  One eventually sees that $q_1 = q_2$ is the only solution since any other solution does not obey Equation \ref{ma1}.  Thus, we have in equilibrium that 
\begin{eqnarray}
q_1(t) = q_2(t) = m_1(t) = m_2(t) = \frac{\lambda}{2 \mu} \quad \mathrm{ as \ } t \to \infty.
\end{eqnarray}

\paragraph{Understanding the stability of the equilibrium} Now that we know the equilibrium for Equations \ref{ma1} - \ref{ma2}, we need to understand the stability of the delay differential equations around the equilibrium.  The first step in doing this is to set each of the queue lengths to the equilibrium values plus a perturbation.  Thus, we set each of the queue lengths to 
\begin{eqnarray}
q_1(t) &=& \frac{\lambda}{2 \mu} + u_1(t) \label{sub11} \\
q_2(t) &=& \frac{\lambda}{2 \mu} + u_2(t) \\
m_1(t) &=& \frac{\lambda}{2 \mu} + u_3(t) \\
m_2(t) &=& \frac{\lambda}{2 \mu} + u_4(t) \label{sub22}
\end{eqnarray}
Substitute Equations  \ref{sub11} - \ref{sub22} into Equations  \ref{ma1} - \ref{ma2} and linearize about the point  $u_1(t) = u_2(t) = u_3(t) = u_4(t) = 0$, giving
\begin{eqnarray}
\updot{u}_1 &=& \frac{\lambda}{4} \cdot \left( u_4(t) - u_3(t) \right) - \mu \cdot u_1(t)
\\
\updot{u}_2 &=& \frac{\lambda}{4} \cdot \left( u_3(t) - u_4(t) \right) - \mu \cdot u_2(t)
\\
\updot{u}_3 &=& \frac{1}{\Delta} \cdot \left( u_1(t) - u_1(t - \Delta) \right)
\\
\updot{u}_4 &=& \frac{1}{\Delta} \cdot \left( u_2(t) - u_2(t - \Delta) \right)
\end{eqnarray}

\paragraph{Uncoupling the differential equations} In their current form the delay differential equations for the perturbations do not yield any insight sincce they are coupled together.  However, we can make a simple transformation and the resulting delay differential equations will become uncoupled.  Thus, we apply the following transformation to uncouple the system of equations:

\begin{eqnarray}
v_1(t) &=& u_1(t) + u_2(t) \\
v_2(t) &=& u_1(t) - u_2(t) \\
v_3(t) &=& u_3(t) + u_4(t) \\
v_4(t) &=& u_3(t) - u_4(t)
\end{eqnarray}
which gives
\begin{eqnarray}
\updot{v}_1(t) &=& - \mu \cdot v_1(t)  \label{vee1}
\\
\updot{v}_2(t) &=& -\frac{\lambda}{2} \cdot v_4(t)  - \mu \cdot v_2(t)  \label{vee2}
\\
\updot{v}_3(t) &=& \frac{1}{\Delta} \cdot \left( v_1(t) - v_1(t - \Delta) \right)  \label{vee3}
\\
\updot{v}_4(t) &=& \frac{1}{\Delta} \cdot \left( v_2(t) - v_2(t - \Delta) \right)  \label{vee4}
\end{eqnarray}
The general solution of Equation \ref{vee1} is $v_1=c_1 \exp({-\mu t})$ and is stable (bounded).  This also implies that Equation \ref{vee3} is also stable and bounded since it only depends on the solution of Equation \ref{vee1}.
To study Equations \ref{vee2} and \ref{vee4}, we let 
\begin{eqnarray}
v_2 &=& A \exp(rt) \\
v_4 &=& B \exp(rt) .
\end{eqnarray}
These solutions imply the following relationships between the constants A,B, and r.

\begin{eqnarray}
A r  &=& - \frac{\lambda}{2} B - \mu A \\
B r &=& \frac{1}{\Delta} ( A - A \exp( -r \Delta))
\end{eqnarray}

solving for A yields

\begin{eqnarray}
A   &=& - \frac{\lambda}{2(\mu + r )} B 
\end{eqnarray}
and rearranging yields the following equation for $r$
\begin{eqnarray} \label{r-trans2}
r &=& \frac{\lambda}{2 \Delta \cdot r} ( \exp( -r \Delta) - 1) - \mu .
\end{eqnarray}
Now it remains for us to understand the transition between stable and unstable solutions once again.

\paragraph{Understanding the transition between stable and unstable solutions} 
\vspace{.1in}
To find the transition between stable and unstable solutions, set $r=i\omega$, giving us the following equation
\begin{equation}
i\omega = \frac{\lambda}{2 \Delta i \omega}( \exp(-i \omega \Delta) -1)-\mu.
\end{equation}
Multiplying both sides by $i \omega$ and using Euler's identity, we have that 
\begin{equation} \label{maeqniw}
 \frac{\lambda}{2 \Delta }( \cos( \omega \Delta) - i \sin(\omega \Delta) -1)-\mu i \omega + \omega^2 = 0.
\end{equation}
Writing the real and imaginary parts of Equation \ref{maeqniw}, we get:
\begin{equation}
\cos(\omega \Delta) = 1 - \frac{2 \Delta \omega^2}{\lambda}
\end{equation}
for the real part and 
\begin{equation}\label{sineqnma}
\sin(\omega \Delta) = - \frac{2 \Delta \mu \omega}{\lambda}
\end{equation}
Once again by squaring and adding $\sin\omega \Delta$ and $\cos\omega \Delta$  together, we get:
\begin{equation}
\omega = \sqrt{\frac{\lambda}{\Delta} - \mu^2}
\end{equation}

Finally by substituting the expression for $\omega$ into Equation \ref{sineqnma} gives us the final expression for the critical delay, which is the solution to the following transcendental equation:
\begin{equation}\label{crit-2}
 \fbox{ $\displaystyle{  \sin\left( \Delta \cdot \sqrt{\frac{\lambda}{\Delta} - \mu^2} \right) + \frac{2 \mu \Delta}{\lambda} \cdot \sqrt{\frac{\lambda}{\Delta} - \mu^2} = 0 }  $}.
\end{equation}

\end{proof}
\end{theorem}

\begin{figure}
\captionsetup{justification=centering}
		\hspace{-.35in}~\includegraphics[scale=.22]{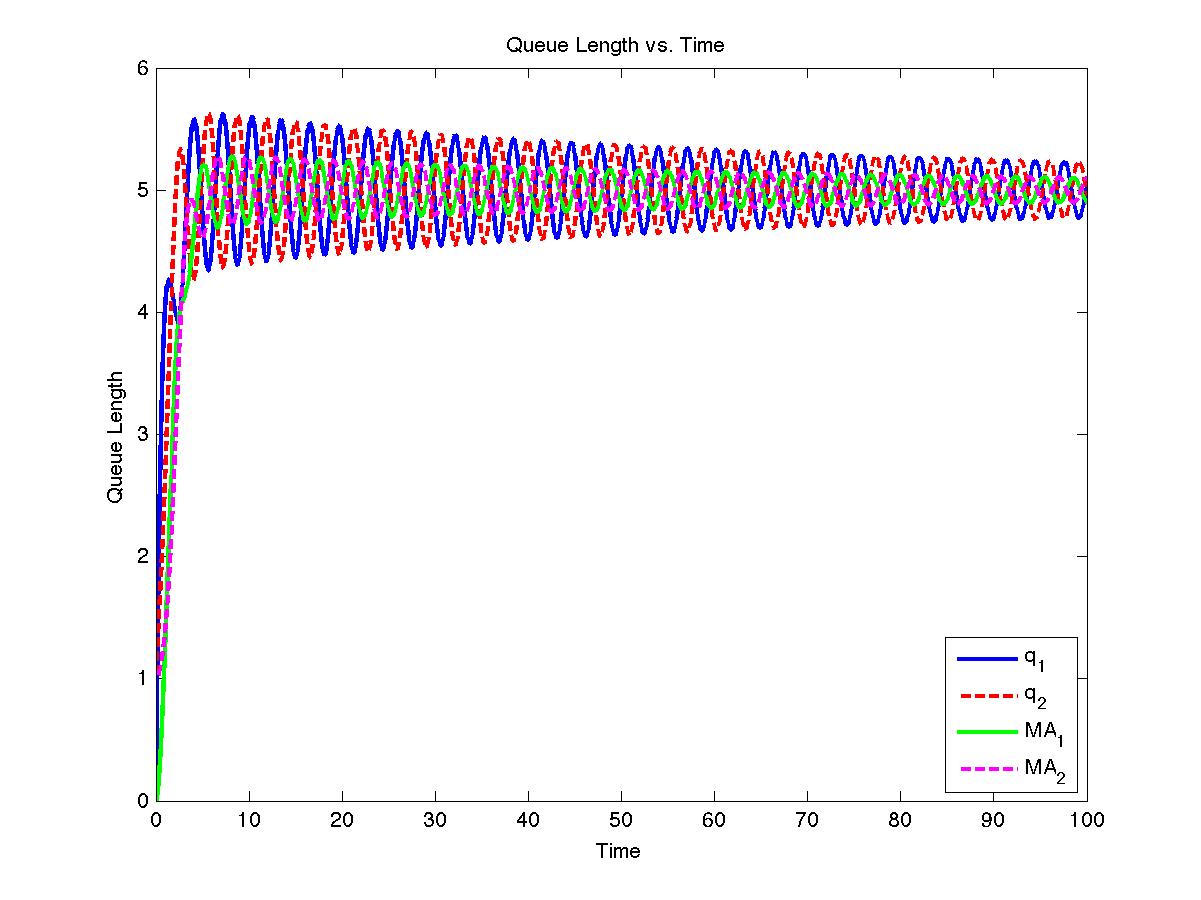}~\hspace{-.3in}~\includegraphics[scale=.22]{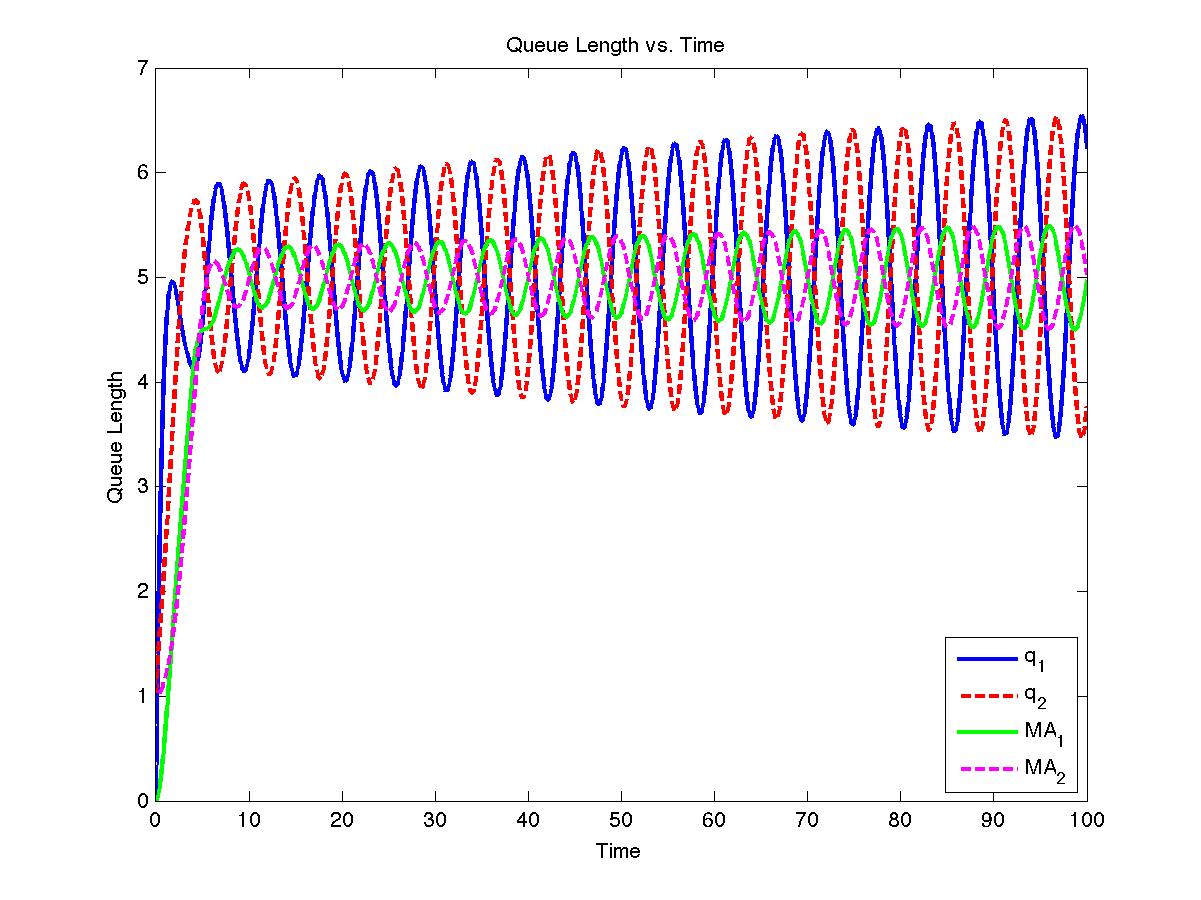} 
	\caption{ $\lambda=10$, $\mu=1$,  \\ $\Delta = 2$  (Left) \quad  $\Delta = 4$ (Right).}\label{Fig9}
\end{figure}

In Figure \ref{Fig9}, we plot the dynamics for the moving average model when the $\Delta =2$ and when $\Delta=4$. When $\Delta=2$ we see that the dynamics are stable and the queues will converge to the equilibrium point.  However, when $\Delta=4$ we see that the dynamics are not stable and the queues are not synchronized.  These dynamics are the same in Figure \ref{Fig10}.  Like in the constant delay example, we see that the stability of the queues is also given by the first Hopf curve. To the left and bottom of the Hopf curve, the two queues will eventually converge to there equilibrium values; however, to the right and above the Hopf curve, the two queues will be forever asynchronous. 

\begin{figure}
\captionsetup{justification=centering}
		\hspace{-.35in}~\includegraphics[scale=.22] {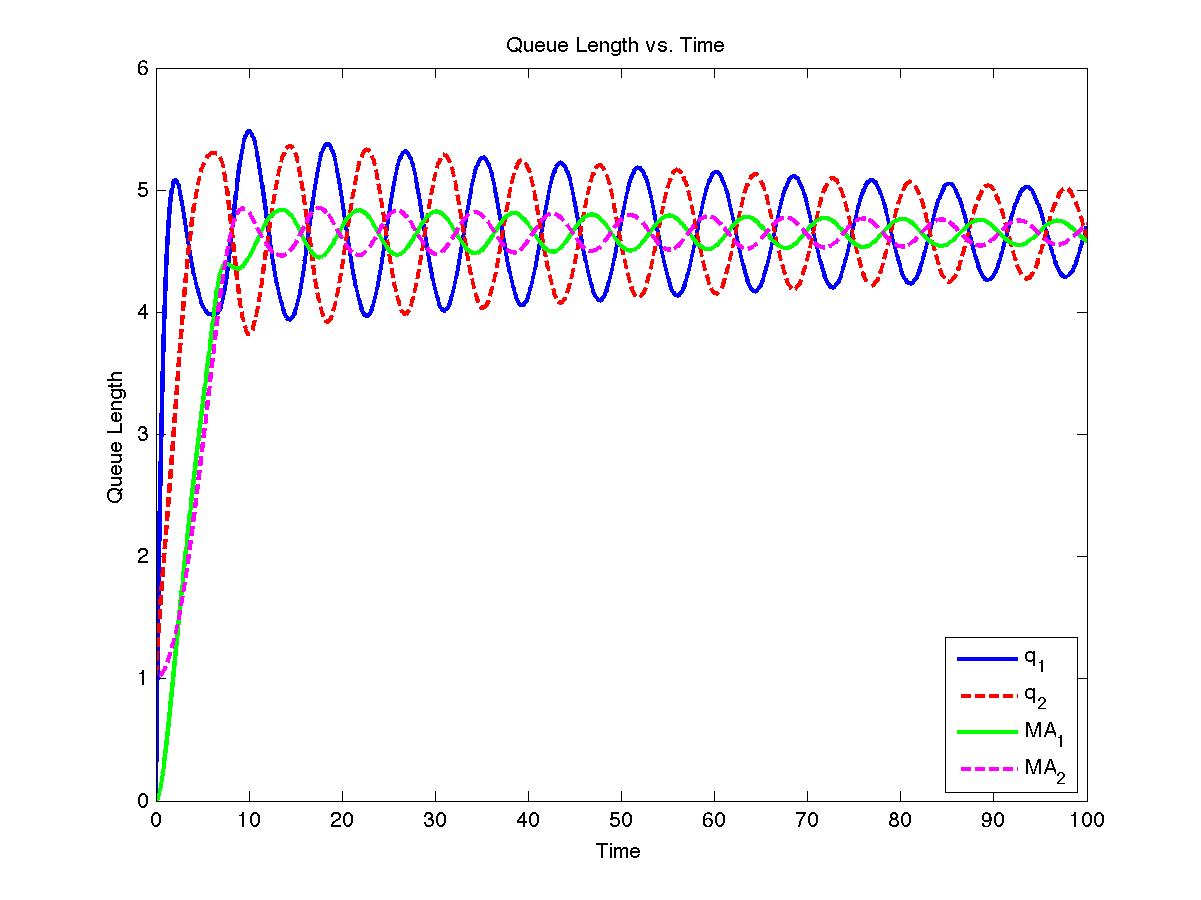}~\hspace{-.3in}~\includegraphics[scale=.22]{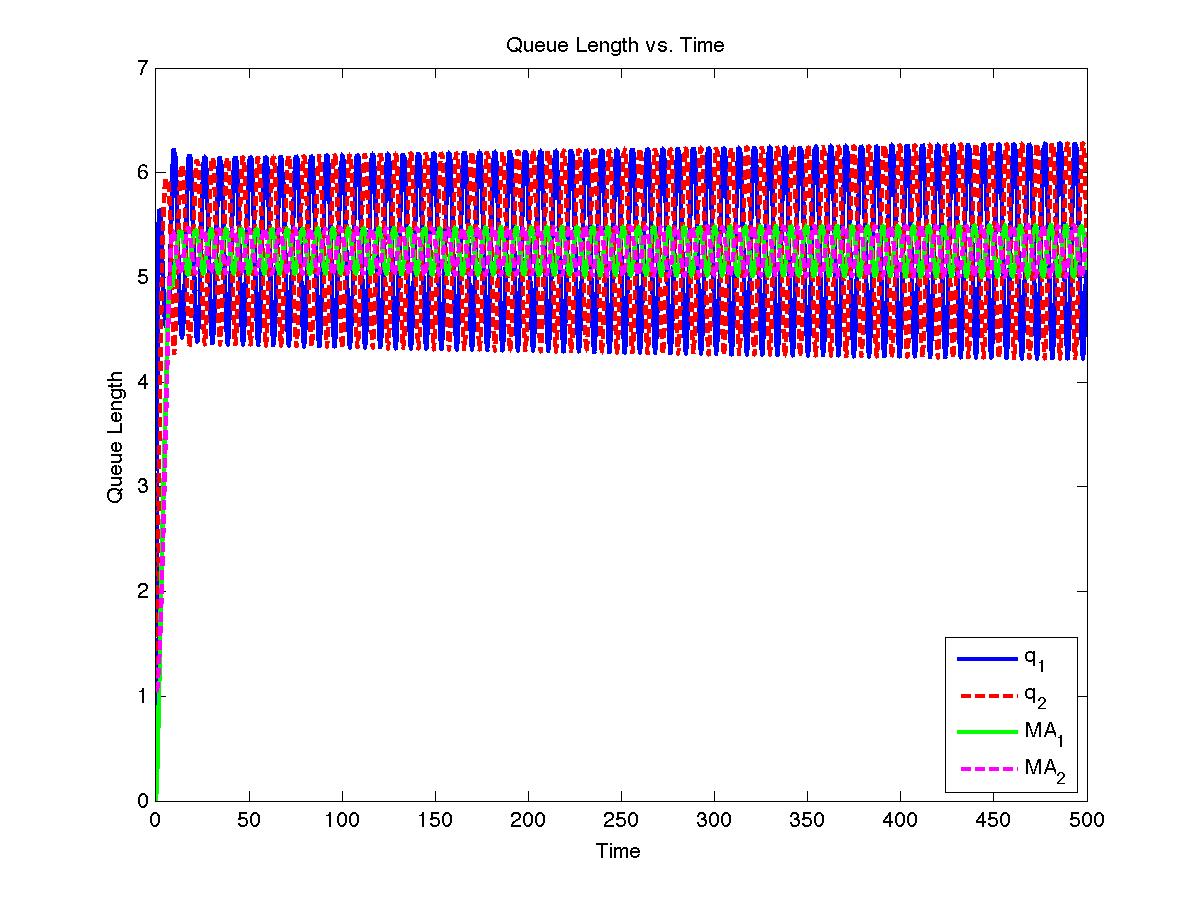} 
	\caption{ $\lambda = 9.3$, $\Delta = 6.5$ ( Left) \quad   \quad \quad  $\lambda = 10.5$, $\Delta = 6.5$ (Right) \\ }\label{Fig10}
\end{figure}

\subsection{Hopf Bifurcation Curves in the Moving Average Model}

Numerical integration with the Matlab delay differential equation package \textbf{dde23} showed that there was only one limit cycle observed in the moving average model, no matter how many of the Hopf curves are crossed. However, like in the constant delay setting, this does not tell us what happens to the various pairs of imaginary roots which occur on the various Hopf curves. We will show that as a perturbation is made increasing the size of the delay, the roots always pass from the left to the right, which implies that the equilibrium remains unstable as the delay increases forever.

Now suppose that the delay $\Delta$ is close to a critical value for a Hopf bifurcation.  We then make a slight perturbation from being near the critical value of the Hopf bifurcation i.e
\begin{equation}
\Delta = \Delta_0 + \epsilon \Delta_1
\end{equation}
where $\epsilon \ll 1$.  When the root $r$ is close to the pure imaginary value, the critical delay is near the value  $\Delta = \Delta_0$.  Thus, when we make a slight perturbation the real and imaginary parts of the root will also be slightly perturbed i.e
\begin{equation}\label{rcd}
r = i \omega + \epsilon ( i r_1 + r_2). 
\end{equation}

\begin{proposition}
For the moving average model, suppose that we make a slight perturbation on the order of $\epsilon \Delta_1$ near a critical delay value for a Hopf bifurcation, then the real part of $r$ is equal to 
\begin{equation}\label{rcd2}
r_2 = \frac{2 \Delta_1 \omega^2 \cdot(2 \Delta_0 \omega^2 - 2 \mu \lambda)}{8 \Delta_0^2 \mu \omega^2 + 12 \Delta_0 \omega^2 + 4 \Delta_0 \lambda \mu + \Delta_0 \lambda^2 + 4 \lambda }.
\end{equation}
\end{proposition}

\begin{proof}
In order to prove this, we can follow the same steps as in Proposition \ref{prop1}.  When $\epsilon = 0$, we reduce back to the original critical threshold of Equation \ref{crit-2}.  Now we substitute Equation \ref{rcd2} into Equation \ref{r-trans2} and do a Taylor expansion for small values of $\epsilon$.  Then solve for $r_1$ and $r_2$.  Solving for the real part of $r$, we find that $r_2$ is equal to the following value  
\begin{equation}\label{rcd}
r_2 =  \frac{2 \Delta_1 \omega^2 \cdot(2 \Delta_0 \omega^2 - 2 \mu \lambda)}{8 \Delta_0^2 \mu \omega^2 + 12 \Delta_0 \omega^2 + 4 \Delta_0 \lambda \mu + \Delta_0 \lambda^2 + 4 \lambda }.
\end{equation}

\end{proof}

In the moving average model, we see that the real part of the roots has the same sign as $\Delta_1$ when $\Delta_0 \omega^2 > \mu \lambda$ and has the opposite sign when $\Delta_0 \omega^2 < \mu \lambda$.  Therefore, when $\Delta_0 \omega^2 > \mu \lambda$, the roots move from the left to the right and the stability is preserved.  However, when $\Delta_0 \omega^2 < \mu \lambda$ this may not be the case.  This is caused by the fact that the Hopf curve is not monotone as a function of $\lambda$ and can be seen in Figure \ref{Fig8}.

\section{Conclusion and Future Research} \label{sec_conclusion}

In this paper, we analyze two new two-dimensional fluid models that incorporate customer choice and delayed queue length information.  The first model considers the customer choice as a multinomial logit model where the queue length information given to the customer is delayed by a constant $\Delta$.  We derive an explicit threshold for the critical delay where below the threshold the two queues are balanced and converge to the equilibrium.  However, when $\Delta$ is larger than the threshold, the two queues have asynchronous dynamics and the equilibrium point is unstable.  In the second model, we consider customer choice as a multinomial logit model where the queue length information given to the customer is a moving average over an interval of $\Delta$.  We also derive an explicit threshold where below the threshold the queues are balanced and above the threshold the queues are asynchronous.  It is important for businesses and managers to determine and know these thresholds since using delayed information can have such a large impact on the dynamics of the business.  Even small delays can cause oscillations and it is of great importance for managers of these service systems to understand when oscillations can arise based on the arrival and service parameters.   

Since our analysis is the first of its kind in the queueing and operations management literature, there are many extensions that are worthy of future study.  One extension that we would like to explore is the impact of nonstationary arrival rates.  This is important not only because arrival rates of customers are not constant over time, but also because it is important to know how to distinguish and separate the impact of the time varying arrival rate from the impact of the delayed information given to the customer.  Other extensions include the use of different customer choice functions, extensions to multiserver queues, and incorporating customer preferences in the model.  A detailed analysis of these extensions will provide a better understand of what information and how the information that operations managers provide to their customers will affect the dynamics of the system.  We plan to explore these extensions in subsequent work.

%


\bibliographystyle{plainnat}
\bibliography{choice}

\begin{thebibliography}{29}
\providecommand{\natexlab}[1]{#1}
\providecommand{\url}[1]{\texttt{#1}}
\expandafter\ifx\csname urlstyle\endcsname\relax
  \providecommand{\doi}[1]{doi: #1}\else
  \providecommand{\doi}{doi: \begingroup \urlstyle{rm}\Url}\fi

\bibitem[Allon and Bassamboo(2011)]{allon2011impact}
Gad Allon and Achal Bassamboo.
\newblock The impact of delaying the delay announcements.
\newblock \emph{Operations research}, 59\penalty0 (5):\penalty0 1198--1210,
  2011.

\bibitem[Allon et~al.(2011)Allon, Bassamboo, and Gurvich]{allon2011we}
Gad Allon, Achal Bassamboo, and Itai Gurvich.
\newblock ``we will be right with you'': Managing customer expectations with
  vague promises and cheap talk.
\newblock \emph{Operations research}, 59\penalty0 (6):\penalty0 1382--1394,
  2011.

\bibitem[Armony and Maglaras(2004)]{armony2004customer}
Mor Armony and Constantinos Maglaras.
\newblock On customer contact centers with a call-back option: Customer
  decisions, routing rules, and system design.
\newblock \emph{Operations Research}, 52\penalty0 (2):\penalty0 271--292, 2004.

\bibitem[Armony et~al.(2009)Armony, Shimkin, and Whitt]{armony2009impact}
Mor Armony, Nahum Shimkin, and Ward Whitt.
\newblock The impact of delay announcements in many-server queues with
  abandonment.
\newblock \emph{Operations Research}, 57\penalty0 (1):\penalty0 66--81, 2009.

\bibitem[Armony et~al.(2015)Armony, Israelit, Mandelbaum, Marmor, Tseytlin,
  Yom-Tov, et~al.]{armony2015patient}
Mor Armony, Shlomo Israelit, Avishai Mandelbaum, Yariv~N Marmor, Yulia
  Tseytlin, Galit~B Yom-Tov, et~al.
\newblock On patient flow in hospitals: A data-based queueing-science
  perspective.
\newblock \emph{Stochastic Systems}, 5\penalty0 (1):\penalty0 146--194, 2015.

\bibitem[Dong et~al.(2015)Dong, Yom-Tov, and Yom-Tov]{dong2015impact}
Jing Dong, Elad Yom-Tov, and Galit~B Yom-Tov.
\newblock The impact of delay announcements on hospital network coordination
  and waiting times.
\newblock Technical report, Working Paper, 2015.

\bibitem[Guo and Zipkin(2007)]{guo2007analysis}
Pengfei Guo and Paul Zipkin.
\newblock Analysis and comparison of queues with different levels of delay
  information.
\newblock \emph{Management Science}, 53\penalty0 (6):\penalty0 962--970, 2007.

\bibitem[Guo and Zipkin(2009)]{guo2009impacts}
Pengfei Guo and Paul Zipkin.
\newblock The impacts of customers'delay-risk sensitivities on a queue with
  balking.
\newblock \emph{Probability in the engineering and informational sciences},
  23\penalty0 (03):\penalty0 409--432, 2009.

\bibitem[Hassin(2007)]{hassin2007information}
Refael Hassin.
\newblock Information and uncertainty in a queuing system.
\newblock \emph{Probability in the Engineering and Informational Sciences},
  21\penalty0 (03):\penalty0 361--380, 2007.

\bibitem[Hui and Tse(1996)]{hui1996tell}
Michael~K Hui and David~K Tse.
\newblock What to tell consumers in waits of different lengths: An integrative
  model of service evaluation.
\newblock \emph{The Journal of Marketing}, pages 81--90, 1996.

\bibitem[Hul et~al.(1997)Hul, Dube, and Chebat]{hul1997impact}
Michael~K Hul, Laurette Dube, and Jean-Charles Chebat.
\newblock The impact of music on consumers' reactions to waiting for services.
\newblock \emph{Journal of Retailing}, 73\penalty0 (1):\penalty0 87--104, 1997.

\bibitem[Ibrahim and Whitt(2008)]{ibrahim2008real}
Rouba Ibrahim and Ward Whitt.
\newblock Real-time delay estimation in call centers.
\newblock In \emph{Proceedings of the 40th Conference on Winter Simulation},
  pages 2876--2883. Winter Simulation Conference, 2008.

\bibitem[Ibrahim and Whitt(2009)]{ibrahim2009real}
Rouba Ibrahim and Ward Whitt.
\newblock Real-time delay estimation in overloaded multiserver queues with
  abandonments.
\newblock \emph{Management Science}, 55\penalty0 (10):\penalty0 1729--1742,
  2009.

\bibitem[Ibrahim and Whitt(2011{\natexlab{a}})]{ibrahim2011real}
Rouba Ibrahim and Ward Whitt.
\newblock Real-time delay estimation based on delay history in many-server
  service systems with time-varying arrivals.
\newblock \emph{Production and Operations Management}, 20\penalty0
  (5):\penalty0 654--667, 2011{\natexlab{a}}.

\bibitem[Ibrahim and Whitt(2011{\natexlab{b}})]{ibrahim2011wait}
Rouba Ibrahim and Ward Whitt.
\newblock Wait-time predictors for customer service systems with time-varying
  demand and capacity.
\newblock \emph{Operations research}, 59\penalty0 (5):\penalty0 1106--1118,
  2011{\natexlab{b}}.

\bibitem[Ibrahim et~al.(2015)Ibrahim, Armony, and Bassamboo]{ibrahim2015does}
Rouba Ibrahim, Mor Armony, and Achal Bassamboo.
\newblock Does the past predict the future? the case of delay announcements in
  service systems, 2015.

\bibitem[Jennings and Pender(2015)]{jennings2015comparisons}
OB~Jennings and J~Pender.
\newblock Comparisons of standard and ticket queues in heavy traffic.
\newblock \emph{Submitted for publication to Queueing Systems}, 2015.

\bibitem[Jouini et~al.(2009)Jouini, Dallery, and
  Ak{\c{s}}in]{jouini2009queueing}
Oualid Jouini, Yves Dallery, and Zeynep Ak{\c{s}}in.
\newblock Queueing models for full-flexible multi-class call centers with
  real-time anticipated delays.
\newblock \emph{International Journal of Production Economics}, 120\penalty0
  (2):\penalty0 389--399, 2009.

\bibitem[Jouini et~al.(2011)Jouini, Aksin, and Dallery]{jouini2011call}
Oualid Jouini, Zeynep Aksin, and Yves Dallery.
\newblock Call centers with delay information: Models and insights.
\newblock \emph{Manufacturing \&amp; Service Operations Management},
  13\penalty0 (4):\penalty0 534--548, 2011.

\bibitem[Munichor and Rafaeli(2007)]{munichor2007numbers}
Nira Munichor and Anat Rafaeli.
\newblock Numbers or apologies? customer reactions to telephone waiting time
  fillers.
\newblock \emph{Journal of Applied Psychology}, 92\penalty0 (2):\penalty0 511,
  2007.

\bibitem[Pender()]{pendersampling}
Jamol Pender.
\newblock Sampling the functional kolmogorov forward equations for
  nonstationary queueing networks.

\bibitem[Pender(2015{\natexlab{a}})]{pender2015heavy}
Jamol Pender.
\newblock Heavy traffic limits for unobservable queues with clearing times.
\newblock 2015{\natexlab{a}}.

\bibitem[Pender(2015{\natexlab{b}})]{pender2015impact}
Jamol Pender.
\newblock The impact of dependence on unobservable queues.
\newblock 2015{\natexlab{b}}.

\bibitem[Plambeck et~al.(2014)Plambeck, Bayati, Ang, Kwasnick, Aratow,
  et~al.]{plambeck2014forecasting}
Erica Plambeck, Mohsen Bayati, Erjie Ang, Sara Kwasnick, Mike Aratow, et~al.
\newblock Forecasting emergency department wait times.
\newblock Technical report, 2014.

\bibitem[Pruyn and Smidts(1998)]{pruyn1998effects}
Ad~Pruyn and Ale Smidts.
\newblock Effects of waiting on the satisfaction with the service: Beyond
  objective time measures.
\newblock \emph{International journal of research in marketing}, 15\penalty0
  (4):\penalty0 321--334, 1998.

\bibitem[Sarel and Marmorstein(1998)]{sarel1998managing}
Dan Sarel and Howard Marmorstein.
\newblock Managing the delayed service encounter: the role of employee action
  and customer prior experience.
\newblock \emph{Journal of Services Marketing}, 12\penalty0 (3):\penalty0
  195--208, 1998.

\bibitem[Taylor(1994)]{taylor1994waiting}
Shirley Taylor.
\newblock Waiting for service: the relationship between delays and evaluations
  of service.
\newblock \emph{The journal of marketing}, pages 56--69, 1994.

\bibitem[Whitt(1999{\natexlab{a}})]{whitt1999improving}
Ward Whitt.
\newblock Improving service by informing customers about anticipated delays.
\newblock \emph{Management science}, 45\penalty0 (2):\penalty0 192--207,
  1999{\natexlab{a}}.

\bibitem[Whitt(1999{\natexlab{b}})]{whitt1999predicting}
Ward Whitt.
\newblock Predicting queueing delays.
\newblock \emph{Management Science}, 45\penalty0 (6):\penalty0 870--888,
  1999{\natexlab{b}}.

\end{thebibliography}
\end{document}